\documentclass[a4paper,11pt]{article}
\usepackage{makeidx}
\usepackage[utf8]{inputenc}
\usepackage[T1]{fontenc}
\usepackage{amsfonts}
\usepackage{bm}
\usepackage{amssymb}
\usepackage{latexsym}
\usepackage{amsmath}
\usepackage{amscd}
\usepackage{amsthm}
\usepackage{rotating}
\usepackage{graphicx}
\usepackage{pifont}
\usepackage{curves}
\usepackage{fancyhdr}
\usepackage{epsfig}
\usepackage{subfigure}
\usepackage{epic}
\usepackage{alltt}
\usepackage[all]{xy}
\usepackage{appendix}
\usepackage{dsfont}
\usepackage{comment}

\title{Continuous limits of large plant-pollinator random networks and some applications\thanks{This work is dedicated to S. M\'el\'eard. The authors thank Mathilde Dufa\"{y} and Pauline Lafitte-Godillon who participated to the early discussions on this project. They also thank Emmanuel Grenier for his help in the proof of Proposition 4.5. This work is funded by the European Union (ERC, SINGER, 101054787). Views and opinions expressed are however those of the author(s) only and do not necessarily reflect those of the European Union or the European Research Council. Neither the European Union nor the granting authority can be held responsible for them. This work has also been supported by the Chair ``Modélisation Mathématique et Biodiversité" of Veolia Environnement-Ecole Polytechnique-Museum National d’Histoire Naturelle-Fondation X. T.R. and V.C.T. acknowledge the support of the CDP C2EMPI, as well as the French State under the France-2030 programme, the University of Lille, the Initiative of Excellence of the University of Lille, the European Metropolis of Lille for their funding and support of the R-CDP-24-004-C2EMPI project. }}

\author{Sylvain Billiard\footnote{Univ. Lille, CNRS, UMR 8198 - Evo-Eco-Paleo, F-59000 Lille, France. sylvain.billiard@univ-lille.fr}, H\'el\`ene Leman\footnote{Inria, ENSL, UMPA, CNRS UMR 5669, 69364 Lyon, France. helene.leman@inria.fr}, Thomas Rey\footnote{Univ. C\^{o}te d'Azur, CNRS, LJAD, F-06108 Nice, France. thomas.rey@univ-lille.fr} and  Viet Chi Tran\footnote{Univ. Lille, CNRS, Inria, UMR 8524 - Laboratoire Paul Painlevé, F-59000 Lille, France. viet-chi.tran@inria.fr}}

\date{\today}
\numberwithin{equation}{section}

\newcommand{\Co}{\mathcal{C}}
\newcommand{\Lip}{\mbox{Lip}}

\newcommand{\eps}{\varepsilon}

\newcommand{\lbrac}{\left[\!\left[}
\newcommand{\rbrac}{\right]\!\right]}

\def\N{\mathbb{N}}

\def\P{\mathbb{P}}
\def\bP{\mathbf{P}}
\def\A{\mathbf{A}}
\def\R{\mathbb{R}}
\def\E{\mathbb{E}}

\def\ind{{\mathchoice {\rm 1\mskip-4mu l} {\rm 1\mskip-4mu l}
{\rm 1\mskip-4.5mu l} {\rm 1\mskip-5mu l}}}

\def\one{\mathds{1}}

\newcommand{\be} {\begin{equation}}
\newcommand{\ee} {\end{equation}}
\newcommand{\bea} {\begin{eqnarray}}
\newcommand{\eea} {\end{eqnarray}}
\newcommand{\Bea} {\begin{eqnarray*}}
\newcommand{\Eea} {\end{eqnarray*}}

\theoremstyle{plain}
\newtheorem{theorem}{Theorem}[section]
\newtheorem{proposition}[theorem]{Proposition}
\newtheorem{lemma}[theorem]{Lemma}
\newtheorem{corollary}[theorem]{Corollary}
\newtheorem{definition}[theorem]{Definition}
\newtheorem{hyp}[theorem]{Assumption}
\newtheorem{remark}[theorem]{Remark}
\newtheorem{example}[theorem]{Example}

\begin{document}

\maketitle
\begin{abstract}We study a stochastic individual-based model of interacting plant and pollinator species through a bipartite graph: each species is a node of the graph, an edge representing interactions between a pair of species. The dynamics of the system depends on the between- and within-species interactions: pollination by insects increases plant reproduction rate but has a cost which can increase plant death rate, depending on the densities of pollinators. Pollinators reproduction is increased by the resources harvested on plants. Each species is characterized by a trait corresponding to its degree of generalism. This trait determines the structure of the interaction graph and the quantities of resources exchanged between species. Our model includes in particular nested or modular networks. Deterministic approximations of the stochastic measure-valued process by systems of ordinary differential equations or integro-differential equations are established and studied, when the population is large or when the graph is dense and can be replaced with a graphon. The long-time behaviors of these limits are studied and central limit theorems are established to quantify the difference between the discrete stochastic individual-based model and the deterministic approximations. Finally, studying the continuous limits of the interaction network and the resulting PDEs, we show that nested plant-pollinator communities are expected to collapse towards a coexistence between a single pair of species of plants and pollinators.
\end{abstract}


Communities of interacting species, such as plants and pollinators, are systems characterized by a large diversity of both their components (dozens of species are typically involved) and the topology of their interaction networks \cite{delmas2019}. How does the structure of the interaction networks affect the dynamics of species and the stability of communities is a common and long-standing question in the ecological literature e.g. \cite{may72, thebaultfontaine, lever2014sudden}. Addressing this question is challenging because the number of species in a community is large, and so is the number of possible edges. Describing ecological networks from observations thus requires much effort, without being ensured that the inferred networks are representative of the studied communities (interactions can indeed depend on space and time, but also on the abundances of the species which makes them harder or not to observe). From a theoretical point of view, the study of the dynamics and stability of communities most often relies either on the numerical analysis of large systems of ordinary differential equations (ODEs), or on the study of the system near equilibrium in simplified communities (see \textit{e.g.} \cite{may72,may73,barbierarnoldibuninloreau,bunin,cohennewman,fyodorovkhoruzhenko, Takeuchi,tangallesina,thebaultfontaine,lever2014sudden} and \cite{akjoujbarbierclenethachemmaidamassolnajimtran} for a review). This makes the identification of general properties of ecological communities difficult. Overall, despite being more than 50 years old, the question of correctly modelling ecological networks and communities is still challenging. In this paper, we develop a new theoretical framework that could help in addressing the previous questions, based on the fact that networks are hierarchically structured. Interactions can indeed be considered at different scales: between individuals, between species or at the level of the whole community \cite{guimaraes2020}. At all scales however, several studies showed evidence that the structure of the interaction networks is due to trait variations between or within species such as size, symmetry, or time of activity, see e.g. \cite{chamberlain2014,watts2016,villalobos2019,arroyo2021}. Starting from a stochastic individual-based model, our goal is to provide simplifications of an ecological network by deriving the continuous limits of both the population sizes and the interaction graph structured by a trait. 

 We focus on a particular ecological community: plant and pollinator species. The interactions between individuals of each species are modelled by a bipartite random network. Interactions affect the demographic rates of plant and pollinator individuals due to resource exchanges. On the one hand, plants benefit from pollinator visits to increase their reproduction rate. On the other hand, pollinators consume nectar, leaves, pollen, etc. which increases their reproduction or survival rates. However, producing such resources is costly for plants, which can negatively affect plant demographic rates, proportionally to pollinator densities \cite{holland2008,holland2010}.
%The purpose of this paper is to understand the dynamics and resilience of such system allowing populations to be large or not, and considering finite or large and dense networks. 

The interaction rates between plant and pollinator individuals are modelled by a random graph where the species are the nodes of the graph. The existence of an edge between a plant and a pollinator species indicates that individuals of these two species can interact. The resulting graph is bipartite, since there is by definition no direct edges between insect species or between plant species. The topology of the graph depends on a trait that represents the degree of generalism of the species. Any pair of plant-pollinator species is connected independently from the other pairs with a probability that depends on the traits of the two considered species. This simple model generalizes Erd\"os-R\'enyi graphs that correspond to the case where the probability of connection is the same for every pairs. Such a model can also generate particular structures such as nested or modular graphs which are commonly found in ecological networks \cite{thebaultfontaine, guimaraes2020}.

%There has been a large literature investigating the dynamics of large networks starting with May \cite{may72,may73}, see e.g. \cite{barbierarnoldibuninloreau,bunin,cohennewman,fyodorovkhoruzhenko, Takeuchi,tangallesina} and \cite{akjoujbarbierclenethachemmaidamassolnajimtran} for a review.
%The dynamics of plant-pollinator networks has been studied by Th\'ebault and Fontaine (e.g. \cite{thebaultfontaine}). %In these articles, the population dynamics directed by the network is described by a (possibly large) system of ordinary differential equations (ODEs). An example are Lotka-Volterra systems of ODEs where the interaction coefficients between two species equal zero if there is no edge in the graph representing species interactions.\\

In Section \ref{sec:stoch-model}, we present the stochastic individual-based model, how the random interaction graph is modelled, and how interactions affect demographic rates. 
Considering large populations but with a fixed number of species, we show in Section \ref{sec:LV} how the dynamics can be approximated by a system of ODEs (close to those commonly studied in the ecological literature). The fluctuations between the approximated deterministic limit and the stochastic individual-based process are established.\\
In Section \ref{sec:graphon}, we derive a continuous approximation of the random graph when the numbers of plant and pollinator species, say $n$ and $m$, are also large. When the graph is dense (\textit{i.e.} the order of the number of edges is in $O(n\times m)$), the complex random network can be replaced by a continuous object, namely a graphon \cite{borgschayeslovaszsosvesztergombi,lovaszbook}. In this case, the high-dimensional system of ODEs can be replaced by two partial integro-differential equations: one for the plants and the other for the pollinators. The latter equations fall into the broader class of kinetic equations, the most well known and studied being the seminal Boltzmann and Vlasov equations. The use of such mesoscopic scale, namely intermediate between a microscopic agent-based approach  \cite{fourniermeleard}, and an averaged macroscopic one (over space) such as \cite{pouchol2018global}, is not a novelty for modelling  competitive interactions between species \cite{desvillettes2008selection,jabin2011selection,bellomo2015interplay,berardo2020interactions}, 
\textit{i.e.} with a large but not infinite number of species with trait structure. Indeed, such approach can be traced back as early as the book of \cite{roughgarden1978} for the mathematical modeling of evolutionary ecology. 
However, the kinetic treatment of large random networks and their approximations by integro-differential equations is new to our knowledge. Finally, in Section \ref{sec:applications}, the large-time behaviors of the ODE system or of the integro-differential equations are studied and simulations are produced to explore several situations.

\section{Stochastic individual-based model with species interactions}\label{sec:stoch-model}

In this section, we introduce a stochastic individual-based model of plant and pollinator species. In this model, the number of species and the number of individuals in each species are finite. Each species is defined by a trait, for instance a morphological or functional trait, which determines its degree of generalism. This trait shapes the interaction network: generalist plant species can be visited by a large number of insect species and generalist pollinator species can visit a large number of plant species. Even though several coevolved traits can be considered in the model (\textit{e.g.} orchids have flowers with a particular morphology, color, phenology, etc.), we will focus on a single trait for the sake of simplicity. The dynamics is driven by the births and deaths of individuals at random times that depend 1) on their positions in the plant-pollinator network, determined by their species traits, 2) by the weights of the interactions between pairs of species, and 3) by the population sizes of the species. The evolution of these interacting populations is described by a stochastic differential equation (SDE) involving Poisson point measures and an acceptance-rejection algorithm as in the Gillespie algorithm \cite{gillespie76} (see \cite{fourniermeleard,ferriere2009stochastic,metztran} for the mathematical formulation of the Gillespie algorithm by mean of SDEs in cases with interactions). 

% \cite{krausefrankmasonulanowicztaylor}.

\subsection{Description of the plant-pollinator community}

We consider $n$ and $m$ plant and pollinator species, respectively. Each plant species is characterized by a trait $x\in [0,1]$ and each pollinator species by a trait $y\in[0,1]$. For $i\in \{1,\dots,n\}=\lbrac 1, n\rbrac$ and $j\in \{1,\dots,m\}=\lbrac 1, m\rbrac$, we respectively denote by $x^i$ and $y^j$ the traits of the plant species $i$ and of the pollinator species $j$.  These traits can represent for instance their degrees of generalism, \textit{i.e.} their tendencies to interact with a large number of other species: a  species is considered as generalist when  its trait ($x^i$ or $y^j$) is close to $1$, or specialist when it is close to $0$.

\subsection{The plant-pollinator interaction network as a bipartite random graph}

Plants and pollinators interact through a bipartite network where each species is a vertex, and an interaction is an edge. We denote $i \sim j$ or $j\sim i$ when individuals of the plant species $i$ can interact with individuals of the pollinator species $j$. There are no edges between two plant or two pollinator species because the network only represents interactions between plants and pollinators. Yet, when specifying the birth and death rates, competition kernels within and between plant species, and within and between pollinator species, will be introduced in the model.

The bipartite graph can be represented by an $n\times m$ adjacency matrix $G^{n,m}$, with 
$$
G^{n,m}_{ij}=\left\{
\begin{aligned}
1 & \text{ if the plant species $i$ can interact with the pollinator species $j$},\\
0 & \text{ otherwise}.
\end{aligned}
\right.
$$
Regarding the structure of the network, we consider a stochastic bipartite graph that generalizes the Erd\"os-R\'enyi random graph (\cite{bollobas2001,durrett,vanderhofstad}). The probability $\phi(x^i,y^j)$ that there is an edge $i \sim j$ between the plant species $i$ and the pollinator species $j$ is assumed to depend only on the traits $x^i$ and $y^j$ of these two species. Different pairs are supposed to be connected independently, the interaction graph is thus built such that:
\begin{hyp}\label{hyp:Gij}
We assume that there exists a continuous function $\phi: [0,1]^2 \to [0,1]$, such that {$\phi$ is Lipschitz continuous with respect to (w.r.t.) both variables and such that} conditionally on the traits $(x^{i},y^{j})_{i,j\in \lbrac 1,n\rbrac \times \lbrac 1,m\rbrac}$, the random variables $(G^{n,m}_{ij})_{(i,j)\in \lbrac 1,n\rbrac\times \lbrac 1,m\rbrac}$ are independent and distributed as Bernoulli random variables with parameters $(\phi(x^i,y^j))_{(i,j)\in \lbrac 1,n\rbrac\times \lbrac 1,m\rbrac}$.
\end{hyp}

%The degree $d^i_P$ (resp. $d^j_A$) of a plant (resp. insect) species is the number of pollinator (resp. plant) species with which the considered species is linked. In the sequel, two types of random graphs will be considered.

\begin{figure}[!ht]
\begin{center}
\begin{tabular}{cc}
\hspace{-0.8cm}\includegraphics[height=6cm,angle=0]{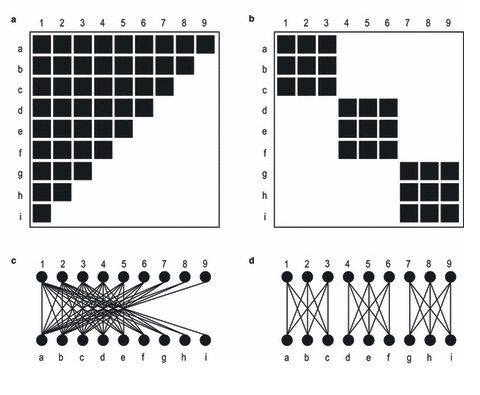} & \includegraphics[height=5.2cm,angle=0,trim=0cm -1cm 0cm 1cm]{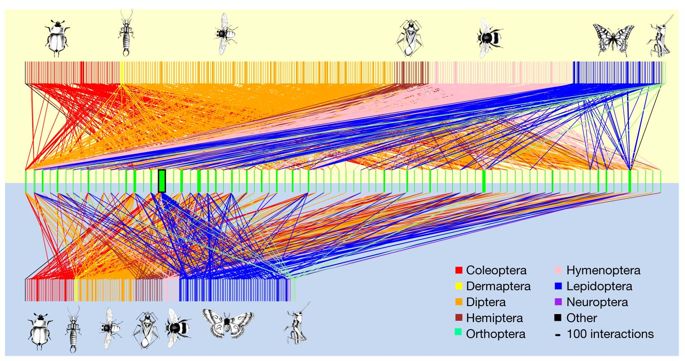} \\
(a) & (b)
\end{tabular}
\caption{{\small \textit{(a): Nested (left) or modular (right) bipartite networks, from Fontaine et al. \cite{fontaineguimareskefithebault}. (b) Pollination network for diurnal and nocturnal insect species, from Knop et al. \cite{knopzollerrysergerpehorlerfontaine}}}}\label{fig:colin}
\end{center}
\end{figure}

\begin{example}\label{exemples:cases}

\noindent \textbf{Case 1.} If there exists $\phi_0\in (0,1)$ such that for all $i$ and $j$, $\phi(x^i,y^j)=\phi_0$, then the stochastic network is a bipartite Erd\"{o}s-R\'{e}nyi graph. 
The degree distributions of nodes representing pollinators (resp. plants), \textit{i.e.} the numbers of edges, do not depend on the species traits. Degrees are thus independent binomial r.v. with parameters $(m,\phi_0)$ (resp. $(n,\phi_0)$) and the total number of edges follows a binomial distribution with parameters $(nm,\phi_0)$.\\
\smallskip
\noindent \textbf{Case 2.} The probability of an edge $i \sim j$ between two species $i$ and $j$ depends on their degree of generalism, for instance by assuming:
\begin{equation}
\phi(x^i,y^j)=x^i y^j.
\end{equation}
Under this assumption, two generalist species ($x^i$ and $y^j$ close to $1$) have a higher probability to be connected than two specialist species ($x^i$ and $y^j$ close to $0$). Nested graphs result from such an assumption (see Fig. \ref{fig:colin} a-left for an illustration of nested graphs definition).\\
\smallskip
\noindent \textbf{Case 3.} When plant and pollinator species can be structured in groups, for instance because of pollination syndrome or spatio-temporal segregation, the probability of an edge $i \sim j$  $\phi(x^i,y^j)=\phi_{IJ}$ depends only on classes $I\ni i$ and $J\ni j$. The resulting random graph is a Stochastic Block Model (SBM \cite{hollandlaskeyleinhardt}, see \textit{e.g.} \cite{abbe} for a review), often called modular networks in community ecology (see Fig. \ref{fig:colin} a-right for an illustration of modular graphs definition). In this case, the function $\phi$ is not Lipschitz continuous, which is a technical assumption needed in Section \ref{sec:graphon} but it is possible to smooth it in this case.
\end{example}

If species $i$ and $j$ interact, \textit{i.e.} if $G^{n,m}_{ij}=1$, we denote by $c_{ij}^{n,m}$ the weight of the interaction $i\sim j$. The quantity $c_{ij}^{n,m}$ describes the intensity and frequency of the relation between the pollinator species $j$ and the plant species $i$. 
From the point of view of the plant, $c_{ij}^{n,m}$ can be interpreted as a measure of the pollination services received from the pollinators. From the point of view of the pollinators, $c_{ij}^{n,m}$ measures the quantity and quality of nutrients collected from the plants. 

\begin{hyp}\label{hyp:cij}
For all $n,m\in \N^2$, conditionally on $(x^{i},y^{j})_{(i,j)\in \lbrac 1,n\rbrac\times \lbrac 1,m\rbrac}$, $(c_{ij}^{n,m})_{(i,j)\in \lbrac 1,n\rbrac\times \lbrac 1,m\rbrac}$ is assumed to be a sequence of independent non-negative r. v. such that 
\begin{itemize}
    \item the expected values only depend on the traits of the plants and pollinators through a function $c^{n,m}\ : [0,1]^2\rightarrow \R$: for all $i\in \lbrac 1, n\rbrac,j\in \lbrac 1,m\rbrac$:
\begin{equation}
\E_{x,y}\left[c^{n,m}_{ij}\right]=c^{n,m}\big(x^i,y^j\big),\label{def:Cijnm}\end{equation}
with $\E_{x,y}[A]=\E[A|(x^{i},y^{j})_{(i,j)\in \lbrac 1,n\rbrac\times \lbrac 1,m\rbrac}]$ for any bounded random variable $A$;
\item there exists $K_c>0$ such that almost surely, 
\[\sup_{n\in\N,m\in\N} \sup_{(i,j)\in \lbrac 1,n\rbrac\times \lbrac 1,m\rbrac}(m+n)c^{n,m}_{i,j}<K_c,\]
in particular, this implies that 
\begin{equation}\label{var:Cijnm}
    V_{max}:=\sup_{n\in\N,m\in\N} \sup_{(i,j)\in \lbrac 1,n\rbrac\times \lbrac 1,m\rbrac}\text{Var}_{x,y}((m+n)c^{n,m}_{i,j})<\infty,
    \end{equation}
where $\text{Var}_{x,y}[A]=\text{Var}[A|(x^{i},y^{j})_{(i,j)\in \lbrac 1,n\rbrac\times \lbrac 1,m\rbrac}]$.
\end{itemize}
\end{hyp}

\begin{example} \label{ex_eco} Depending on plants and pollinators' traits considered, different functions $c^{n,m}$ be considered:
\begin{itemize}
    \item The intensity and frequency of the interaction can be independent of the species traits, i.e. $c^{n,m}_{i,j}=\frac{c}{n+m}$, which would mean that given a plant and a pollinator species interact, the effect of the interaction on their demographic rates is the same for all species;
    \item $x^i$ and $y^j$ can represent the time of peak activity during a day or a season for plants $i$ and pollinators $j$, and $c^{n,m}_{i,j}$ a measure of the overlap of activities for both species; Similarly,  $c^{n,m}_{i,j}$ can represent geographical proximity rather than time overlap between both species;
    \item $x^i$ and $y^j$ can virtually represent some aspects of the morphology of plants' flowers and pollinators' bodies and the r.v. $c^{n,m}_{i,j}$ some morphological preferences. For instance, assuming $(x^i,y^j)\in [0,1]^2$, $c^{n,m}_{i,j}=\frac{x^i y^j}{n+m}$ would mean that pollinators and plants tend to interact more when the traits are both large, while assuming $c^{n,m}_{i,j}=\frac{x^i (1-y^j)}{n+m}$ would mean that they will interact more when the plant has a large trait value and the pollinator a low trait value.
\end{itemize}
\end{example}

\subsection{Stochastic dynamics of the plant-pollinator community}

The dynamics of plant and pollinator populations within the community is ruled by random point birth and death events. We assume that the size of the population of plants and pollinators is scaled by a factor $K>0$, called the carrying capacity in Ecology. This carrying capacity $K$ is a measure of the size of the system, in other words, it controls the total abundance of the whole community that can be sustained by the environment. A continuous limit of the stochastic dynamics will be obtained when the species abundances tend to infinity, in other words when $K \rightarrow \infty$. 

Given the scaling factor $K$, we denote $P^{K,i}_t$ and $A^{K,j}_t$ the sizes of the plant and pollinator species $i$ and $j$ at time $t$. The plant and pollinator populations at time $t$ can be represented by the following point measures:
\begin{equation}
\bP^{K,n,m}_t(dx)=\frac{1}{nK}\sum_{i=1}^n P^{K,i}_t \delta_{x^i}(dx),\qquad \A^{K,n,m}_t(dy)=\frac{1}{mK}\sum_{j=1}^m A^{K,j}_t \delta_{y^j}(dy).\label{def:mesuresAP}
\end{equation}

%%%% Chi : les conditions initiales seront discutées plus tard dans les théorèmes limites
%
%$P^{K,i}$ and $A^{K,j}$ have initial conditions $P^{K,i}_0$ and $A^{K,j}_0$ such that there exist, for all $i\in \{1,\dots n\}$ and $j\in \{1,\dots m\}$, nonnegative real numbers ${P}^i_0$ and ${A}^j_0$ for which
%\begin{equation}
%\lim_{K\rightarrow +\infty} \frac{P^{K,i}_0}{K}={P}^i_0,\qquad \lim_{K\rightarrow +\infty} \frac{A^{K,j}_0}{K}={A}^j_0.
%\end{equation}

The interactions between plants and pollinators are characterized by exchanges of resources \cite{holland2010}. The quantities of resources gained by plants or pollinators $R$ are modeled through a "mass-action model".
%Population growth rates of plants and pollinators, denoted by $g^A(R)$ and $g^P(R)$, are assumed functions of the quantity $R$ of resources obtained from the interactions. 
At time $t$, in the population scaled by $K$, a single individual of the plant species $i$ interacting with pollinator species $j$ is supposed to gain a quantity of resources (here the pollination service) proportional to the abundance of pollinators $A^{K,j}_t/K$ weighted by the interaction efficiency $c_{ij}^{n,m}$,
such that the total resource gained by a plant individual of species $i$ through the pollination interactions is
\begin{equation}\label{Rij:model2}
R^{A,K,i}_t:=\sum_{j\sim i}\frac{c^{n,m}_{ij}  A^{K,j}_t}{K}.
\end{equation}

Similarly, for a given pollinator of the species $j$, the resources gained from the interactions with the plant species %$i$ 
at time $t$ in the population parameterized by $K$ is assumed to be:
\begin{equation}\label{Rij:model3}
R^{P,K,j}_t=\sum_{i\sim j}\frac{c^{n,m}_{ij}  P^{K,i}_t}{K}.
\end{equation}

The dynamics of the community is supposed to be governed by the birth and death rates of the plant and pollinator populations. We denote $b^P(R)$ and $b^A(R)$ the individual birth rate of plant and pollinator species, respectively, each depending on the quantities of resources exchanged $R^P$ or $R^A$. Similarly, we denote $d^P(R)$ and $d^A(R)$ the individual death rates. Finally,  
\begin{equation}g^P(R):=b^P(R)-d^P(R) \quad \text{and} \quad g^A(R):=b^A(R)-d^A(R),\label{def:growthrates}
\end{equation}
are the component of the growth rate due to the interactions between plants and pollinators.  

The plants and pollinators dynamics are also assumed to be affected by logistic competition among plants and among pollinators (within and between species  competition). 
We suppose that competition strength depends on the traits $x$ and $y$. This can represent the fact that plants or pollinators with similar traits tend to share similar ecological niche, or phenology, etc. A plant with trait $x\in [0,1]$ suffers an additional death rate term due to competition such that
\begin{equation}\label{logistic:plant}
k\star \mathbf{P}^{K,n,m}_t(x):=\int_{[0,1]} k(x,x') d\mathbf{P}^{K,n,m}_t(dx')=\frac{1}{nK}\sum_{i=1}^n k(x,x^i) P^{K,i}_t,
\end{equation} 
where $k(x,x')$ quantifies the competition pressure exerted by another plant of trait $x'$. Similarly, a pollinator with trait $y' \in [0,1]$ suffers an additional death rate due to competition with pollinators of trait $y$ given by
\begin{equation}\label{logistic:pollinator}
h\star \mathbf{A}^{K,n,m}_t(y):=\int_{[0,1]} h(y,y') d\mathbf{A}^{K,n,m}_t(dy')=\frac{1}{mK}\sum_{j=1}^m h(y,y^j) A^{K,j}_t,
\end{equation}
where $h(y,y')$ quantifies the competition of individuals with trait $y'$ on individuals with trait $y$.

The following assumptions on the functions involved in the model will be needed, both for modelling and mathematical purposes.
\begin{hyp}\label{hyp_functions}
(i) The birth rates $b^P$ and $b^A$ are assumed to be bounded on $\R^+$ by constants $M^P>0$ and $M^A>0$ respectively. Moreover, all the rate functions $b^P$, $b^A$, $d^P$ and $d^A$ are assumed to be locally Lipschitz continuous on $\R^+$. \\
(ii) The competition kernels $k$ and $h$ are assumed to be Lipschitz continuous w.r.t. both variables on $[0,1]^2$. 
\end{hyp}

\subsection{Stochastic differential equations}

Following works by M\'el\'eard and co-authors \cite{fourniermeleard,champagnatferrieremeleard}, it is possible to describe the evolution of the population measures $(\bP^{K,n,m}_t,\A^{K,n,m}_t)_{t\in \R_+}$ defined in \eqref{def:mesuresAP} by SDEs driven by Poisson point measures. Let us first present these SDEs and then explain their heuristic meaning.

\begin{definition}\label{def_processes}
Suppose that assumptions \ref{hyp:Gij}, \ref{hyp:cij} and \ref{hyp_functions} hold. 
Let $Q_B^P(ds,dk,d\theta)$, $Q_D^P(ds,dk,d\theta)$, (resp. $Q_B^A(ds,dk,d\theta)$ and $Q^A_D(ds,dk,d\theta)$) be Poisson point measures on $\R_+\times E:=\R_+\times \lbrac 1,n\rbrac \times \R_+$ (resp. $\R_+\times F:= \R_+\times \lbrac 1,m \rbrac \times \R_+$) with intensity measure $q(ds,dk,d\theta)=ds\, \mathfrak{n}(dk)\, d\theta$ where $ds$ and $d\theta$ are Lebesgue measures on $\R_+$ and where $\mathfrak{n}(dk)$ is the counting measure on $\N^*=\{1,2,\dots\}$. {Consider initial conditions $P^{K,i}_0$'s and $A^{K,j}_0$'s that are independent from the Poisson point measures and from the $G_{ij}^{n,m}$'s and $c_{ij}^{n,m}$'s.} 

\noindent For $i\in \lbrac 1,n\rbrac$ and $j\in \lbrac 1,m\rbrac$, one has
\begin{align}
P^{K,i}_t=  & P^{K,i}_0+\int_0^t \int_{E}  \ind_{i=k} \ind_{\theta\leq  b^P\big(R^{A,K,i}_{s-}\big)P^{K,i}_{s-}} Q^P_B(ds,dk,d\theta)\nonumber\\
 & \hspace{1cm} -\int_0^t \int_{E}  \ind_{i=k} \ind_{\theta\leq \big[d^P\big(R^{A,K,i}_{s-}\big)+k\star\mathbf{P}^{K,n,m}_{s-}(x^i) \big]P^{K,i}_{s-}} Q^P_D(ds,dk,d\theta)\label{def_Poisson1}\\
A^{K,j}_t= & A^{K,j}_0+ \int_0^t \int_{F}  \ind_{j=k} \ind_{\theta\leq b^A\big(R^{P,K,j}_{s-}\big)A^{K,j}_{s-}} Q^A_B(ds,dk,d\theta)\nonumber\\
& \hspace{1cm}-\int_0^t \int_{F}  \ind_{j=k} \ind_{\theta\leq \big[d^A\big(R^{P,K,j}_{s-}\big)+h\star \mathbf{A}^{K,n,m}_{s-}(y^j)\big] A^{K,j}_{s-}} Q^A_D(ds,dk,d\theta).\label{def_Poisson2}
\end{align}
\end{definition}

The above SDEs correspond to the mathematical formulation of the individual-based simulations classically used in ecology and originated with Gillespie's algorithm \cite{gillespie76,gillespie77}. The Poisson point measures $Q_B^P$ and $Q_B^A$ (resp. $Q_D^P$ and $Q_D^A$) give the random times of possible birth (resp. death) events. The indicators $\ind_{i=k}$ or $\ind_{j=k}$ in the integral ensure that the equations for $P^{K,i}$ or $A^{K,j}$ are indeed modified when the birth or death events affect the plant population $i$ or pollinator population $j$. The indicators in $\theta$ correspond to an acceptance-rejection algorithm so that the birth and death times occur with the correct rates. Note that in these SDEs, the network is hidden in the definitions of the terms $R^{A,K,i}_s$ and $R^{P,K,j}_s$ (see \eqref{Rij:model2} and \eqref{Rij:model3}).\\

In order to approximate the individual-based model with ODEs or integro-differential equations, we can reformulate the SDEs \eqref{def_Poisson1} and \eqref{def_Poisson2} by expressing the processes $P^{K,i}$ and $A^{K,j}$ as semi-martingales. The following proposition also states the existence and uniqueness of a solution to \eqref{def_Poisson1}-\eqref{def_Poisson2}. For a finite measure $\mu$ and a positive measurable function $f$, we will denote by $\langle \mu,f\rangle=\int f \, d\mu$ the integral of $f$ w.r.t. $\mu$.

\begin{proposition}\label{prop:stochpro}
Consider here $n$, $m$ and $K$ fixed. If Assumption \ref{hyp_functions} holds and if
\[\E\left[\langle \mathbf{P}^{K,n,m}_0,\one\rangle^{2}+\langle \A^{K,n,m}_0,\one\rangle^{2}\right]<+\infty,\]
then the processes defined by \eqref{def_Poisson1} and \eqref{def_Poisson2} are well defined on $\R_+$. Moreover, for all $T\geq 0$,
\begin{equation}\label{propagation:moment}
\E\left[\sup_{t\in [0,T]}\langle \mathbf{P}^{K,n,m}_t,\one\rangle^{2}+\langle \A^{K,n,m}_t,\one\rangle^{2} \right]<+\infty. 
\end{equation}
For any $f: [0,1]\rightarrow \R$ measurable test function, we have:
\begin{align}
\langle  \mathbf{P}^{K,n,m}_t,f\rangle= & \frac{1}{n}\sum_{i=1}^n \frac{1}{K}P^{K,i}_t f(x^i) \nonumber\\
= & \langle  \mathbf{P}^{K,n,m}_0,f\rangle + \int_0^t \frac{1}{nK}\sum_{i=1}^n f(x^i) \Big[ g^P\big(R^{A,K,i}_{s_-}\big)-k\star \mathbf{P}^{K,n,m}_s(x^i)\Big]  P^{K,i}_s\ ds \nonumber\\
& +  \frac{1}{nK}\sum_{i=1}^n f(x^i) M^{K,i}_t,\label{eq:semimartP}
\end{align}
where $(M^{K,i})_{i\in \{1,..,n\}}$ are square integrable martingales with predictable quadratic variation processes:
\begin{equation}\label{eq:martingaleP}
\langle M^{K,i}\rangle_t= \int_0^t \left(b^P\big(R^{A,K,i}_{s}\big)+ d^P\big(R^{A,K,i}_{s}\big)+k\star \mathbf{P}^{K,n,m}_s(x^i)\right) P^{K,i}_s \ ds.\end{equation}
A similar expression holds for the pollinator populations. For any $f: [0,1]\rightarrow \R$ measurable test function, we have:
\begin{align}
\langle  \mathbf{A}^{K,n,m}_t,f\rangle= & \frac{1}{m}\sum_{j=1}^m \frac{1}{K}A^{K,j}_t f(y^j) \nonumber\\
= & \langle  \mathbf{A}^{K,n,m}_0,f\rangle + \int_0^t \frac{1}{mK}\sum_{j=1}^m f(y^j) \Big[g^A\big(R^{P,K,j}_{s_-}\big) - h\star \mathbf{A}^{K,n,m}_s(y^j)\Big] A^{K,j}_s\ ds\nonumber\\
 &  + \frac{1}{mK}\sum_{j=1}^m f(y^j) M^{K,j}_t,\label{eq:semimartA}
\end{align}
where $(M^{K,j})_{j\in \{1,..,m\}}$ are square integrable martingales with predictable quadratic variation processes:
\[\langle M^{K,j}\rangle_t= \int_0^t \left(b^A\big(R^{P,K,j}_{s}\big)+ d^A\big(R^{P,K,j}_{s}\big)+ h\star \mathbf{A}^{K,n,m}_s(y^j)\right) A^{K,j}_s \ ds.\]
\end{proposition}
The proof follows from usual stochastic calculus with Poisson point processes (\textit{e.g.} \cite{ikedawatanabe}), as developed in \cite{bansayemeleard,fourniermeleard} for example. We however give a sketch of proof in Appendix~\ref{app_proof1}.

Thus the processes $(\bP^{K,n,m}_t)_{t\in \R_+}$ and $(\A^{K,n,m}_t)_{t\in \R_+}$ are well-defined in the set $\mathbb{D}(\R_+,\mathcal{M}_F([0,1]))$ of right-continuous left-limited (càdlàg) processes with values in the set of finite measures on $[0,1]$. The space $\mathcal{M}_F([0,1])$ embedded with the weak topology is a Polish space and the set of càdlàg functions is embedded with the Skorokhod topology which makes it Polish as well (\textit{e.g.} \cite[Th. 6.8 and 12.2]{billingsley99}).

\section{Community dynamics limit when abundances are large but the number of species is fixed}\label{sec:LV}

In this section, we consider the numbers of species $n$ and $m$ as fixed, while parameter $K$ rescaling the population sizes of plants and pollinators tends to $+\infty$. In this section, since $n$ and $m$ are fixed, these letters will be omitted to avoid cumbersomeness when there is no possible confusion.

\subsection{Law of large numbers}

\begin{proposition}\label{prop_lln}
We consider a sequence $(\bP^{K,n,m},\A^{K,n,m})_{K\in \N}$ of processes as in Definition~\ref{def_processes}, with initial conditions such that 
\begin{equation}\sup_{K\in \N} \E\left[\langle \mathbf{P}^{K,n,m}_0,\one\rangle^{3}+\langle \A^{K,n,m}_0,\one\rangle^{3}\right]<+\infty,\label{hyp:moments3}\end{equation}
and such that there exist $(\widetilde{P}^1_0,\dots, \widetilde{P}^n_0)$ and $(\widetilde{A}_0^1,\dots \widetilde{A}_0^m)$, independent from the $G_{ij}^{n,m}$'s and $c_{ij}^{n,m}$'s, satisfying the following convergences almost surely (a.s.) :
$$\lim_{K\rightarrow +\infty} \frac{P^{K,i}_0}{K}=\widetilde{P}^i_0,\qquad \lim_{K\rightarrow +\infty} \frac{A^{K,j}_0}{K}=\widetilde{A}^j_0.$$
Then, for all $T\geq 0$, the following convergence holds a.s.:
\begin{equation}
\lim_{K\to\infty} \sup_{t\leq T} \sup_{i,j} \left\{ \left|\frac{P^{K,i}_t}{K} - \widetilde{P}^i_t  \right|,
\left|\frac{A^{K,j}_t}{K} - \widetilde A^j_t  \right| \right\}=0,
\end{equation}
where $(\widetilde{P}^1_t,\dots, \widetilde{P}^n_t)_{t\geq 0}$ and $(\widetilde{A}_t^1,\dots \widetilde{A}_t^m)_{t\geq 0}$ are the unique solution of the system, which exists on $\R^{n+m}$:
\begin{equation}
\label{edo}
    \begin{aligned}
\forall i\in \lbrac 1, n\rbrac,\quad & \frac{d\widetilde{P}^i_t}{dt}=  \left(g^P\big(  \widetilde R^{A,i}_t\big)   - \frac{1}{n}\sum_{\ell=1}^n k(x^i,x^\ell){\widetilde P}^\ell_t\right)  \widetilde{P}^i_t \\
j\in \lbrac 1, m\rbrac,\quad & \frac{d\widetilde{A}^j_t}{dt}=  \left(g^A\big( \widetilde R^{P,j}_t\big)  -\frac{1}{m} \sum_{\ell=1}^m h(y^j,y^\ell) \widetilde{A}^\ell_t \ \right) \widetilde A^j_t,
    \end{aligned}
\end{equation}
where, for all $t\in\R^+$, $\widetilde{R}^{P,j}_t=\sum_{i\sim j} c_{ij}^{n,m}\widetilde{P}^i_t $ and $\widetilde{R}^{A,i}_t=\sum_{j\sim i} c_{ij}^{n,m}\widetilde{A}^j_t$.
\end{proposition}

Equations \eqref{edo} is similar to a classical Lotka-Volterra system applied to mutualistic interactions with competition. In our case, the species community is structured by the traits  $x^i$ of plants and $y^j$ of pollinators, and these traits determine the probability and strength of the interactions. In addition, how interactions translate into births and deaths (the so-called \textit{numerical response} in ecological terms) is embedded in functions $g^A$ and $g^P$ (demographic growth) and $k$ and $h$ (competitive kernels). These functions can take any form (see Section \ref{sec:applications} for some examples). As a consequence, \eqref{edo} can capture a large variety of ecological situations. In particular, many ODE models published in the ecological literature are special cases of \eqref{edo} (\textit{e.g.} \cite{thebaultfontaine, bascomptejordanomelianolesen,lever2014sudden}) and therefore can be seen as limits of stochastic individual-based models as defined in Section \ref{sec:stoch-model}.  
Notice that coefficients $c_{ij}^{n,m}$ may be random coefficients as explained in the previous section. However, these coefficients remain constant over time.

\begin{proof}Since the numbers $n$ and $m$ of species are supposed constant, working with the vector processes $(P^{K,i}_., A^{K,j}_.)_{i\in \lbrac 1,n\rbrac, j\in  \lbrac 1,m\rbrac}$ with values in $\R_+^{n+m}$ (and not the measure-valued processes) is here sufficient. Notice that under the assumption \eqref{hyp:moments3}, we can obtain by computations similar to the Step 1 of the proof of Proposition \ref{prop:stochpro} that:
\begin{equation}
   \sup_{K\in \N} \E\Big[\sup_{t\in [0,T]} \langle \mathbf{P}^{K,n,m}_t,\one\rangle^3+\langle \mathbf{A}^{K,n,m}_t,\one\rangle^3\Big]<+\infty.
\end{equation}The global existence and uniqueness of solution to \eqref{edo} can be proved by classical results for ODEs: it follows from the local boundedness and Lipschitz property of the functions on the right-hand side. Then this proposition is a direct application of Theorem 2.1 p.456 in the book by Ethier and Kurtz \cite{ethierkurtz}.
\end{proof}

As a consequence:
\begin{corollary}
Under the assumptions of Proposition \ref{prop_lln}, and for any $T>0$, the sequence of measure-valued processes $(\mathbf{P}^{K,n,m},\mathbf{A}^{K,n,m})_{K\in \N}$ converges a.s. and uniformly in the Skorohod space \\
$\mathbb{D}\left ([0,T],\mathcal{M}_F([0,1])^2 \right)$ to the process $({\mathbf{\widetilde P}}^n,{\mathbf{\widetilde A}}^m)$ defined by:
\[{\mathbf{\widetilde P}}^n_t(dx)=\frac{1}{n}\sum_{i=1}^n \widetilde{P}^{i}_t \delta_{x^i}(dx),\qquad {\mathbf{\widetilde A}}^m_t(dy)=\frac{1}{m}\sum_{j=1}^m \widetilde{A}^{j}_t \delta_{y^j}(dy).\] 
\end{corollary}In this corollary, the space $\mathcal{M}_F([0,1])^2$ can be embedded with the total variation topology which is stronger than the weak topology: as $n$, $m$, the sequences $(x^1,\dots x^n)$ and $(y^1,\dots y^m)$ remain unchanged, all the measures are absolutely continuous with respect to the same counting measures $\sum_{i=1}^n \delta_{x^i}$ or $\sum_{j=1}^m \delta_{y^j}$. Also, the uniform convergence in the Skorohod space for any $T>0$ yields the convergence for the Skorohod topology on $\R_+$ (see Theorem 16.2 in~\cite{billingsley99}).\\

For the remainder, let us denote by $\Phi^{P,n,m}=(\Phi^{P,n,m}_1,\cdots \Phi^{P,n,m}_n)$ and $\Phi^{A,n,m}=(\Phi^{A,n,m}_1,\cdots \Phi^{A,n,m}_m)$ the applications from $\R_+^{n+m}$ into $\R^n$ and $\R^m$ respectively such that \eqref{edo} rewrites for all $i\in \lbrac 1,n\rbrac$ and $j\in \lbrac 1,m\rbrac$:
\begin{align}
 \frac{d\widetilde{P}^i_t}{dt}=  \Phi^{P,n,m}_i(\widetilde{P}_t,\widetilde{A}_t) \quad \mbox{ and }\quad
 \frac{d\widetilde{A}^j_t}{dt}=  \Phi^{A,n,m}_j(\widetilde{P}_t,\widetilde{A}_t) .\label{edo2}
\end{align}

\subsection{Central limit theorem}

The individual-based model (mathematically described by the SDE of Definition \ref{def_processes}) is very commonly encountered in simulations in Biology, while the ODEs \eqref{edo} are often used in the modelling papers. The law of large number (Proposition \ref{prop_lln}) shows that both points of view are related in the limit $K\rightarrow +\infty$. In this section, we quantify the speed at which the convergence holds by establishing a central limit theorem. This can be useful for instance to compute confidence intervals or study the statistical properties of parameter estimators.

Let us introduce the fluctuation processes:
\begin{equation}
\eta^{K,P}_t= \sqrt{K}\left(\begin{array}{c}
P^{K,1}_t/K - \widetilde{P}^1_t\\
\vdots \\
P^{K,n}_t/K - \widetilde{P}^n_t
\end{array}\right),\qquad \mbox{ and }
\eta^{K,A}_t= \sqrt{K}\left(\begin{array}{c}
A^{K,1}_t/K - \widetilde{A}^1_t\\
\vdots \\
A^{K,m}_t/K - \widetilde{A}^m_t
\end{array}\right).
\end{equation}

\begin{proposition}\label{prop:tcl}
Consider the process of Definition \ref{def_processes}, assume that the functions $b^A$, $b^P$, $d^A$ and $d^P$ are of class $\Co^1$, and  that the initial conditions $(\eta^{K,P}_0, \eta^{K,A}_0)_{K\geq 0}$ converge in distribution towards a deterministic vector $(\widetilde{\eta}^{P}_0, \widetilde{\eta}^{A}_0)$
when $K\rightarrow +\infty$. Under the same assumptions as Proposition \ref{prop_lln}, one has
$$
(\eta^{K,P}_t, \eta^{K,A}_t)_{t\geq 0} \underset{K\to\infty}{\Rightarrow} (\widetilde\eta^{P}_t, \widetilde\eta^{A}_t)_{t\geq 0},
$$
where the converge holds in distribution in the Skorohod space $\mathbb{D}(\R_+,\R^{n+m})$ and where the processes $(\widetilde\eta^{P}_., \widetilde\eta^{A}_.)=(\widetilde{\eta}^{P,i}_., \widetilde{\eta}^{A,j}_.)_{i\in \lbrac 1,n\rbrac, j\in  \lbrac 1,m\rbrac}$ are solutions of the following SDEs driven by $n+m$ independent standard real-valued Brownian motions $(W^{P,i},W^{A,j})_{i\in \lbrac 1,n\rbrac, j\in  \lbrac 1,m\rbrac}$: for all $i\in \lbrac 1,n\rbrac$ and $j\in \lbrac 1,m\rbrac$,
\begin{equation}
\begin{aligned}
 \widetilde\eta^{P,i}_t= &\widetilde\eta^{P,i}_0 +\int_0^t \sqrt{\bigg(b^P\big( \widetilde R^{A,i}_s\big)+d^P\big( \widetilde R^{A,i}_s\big)+\frac{1}{n}\sum_{\ell=1}^n k(x^i,x^\ell)\widetilde  P^\ell_s\bigg) \widetilde{P}^i_s}\ dW^{P,i}_s \\
& +\int_0^t \left[\sum_{\ell=1}^n \frac{\partial \Phi^{P,n,m}_{i}}{\partial p^\ell} ( \widetilde P_s,  \widetilde A_s)\widetilde\eta^{P,\ell}_s + \sum_{\ell=1}^m \frac{\partial \Phi^{P,n,m}_i}{\partial a^\ell} ( \widetilde P_s,  \widetilde A_s)\widetilde\eta^{A,\ell}_s \right]  ds  ,\\
\widetilde\eta^{A,j}_t= &\widetilde\eta^{A,j}_0 +\int_0^t \sqrt{\bigg(b^A\big( \widetilde R^{P,j}_s\big)+d^A\big( \widetilde R^{P,j}_s\big)+\frac{1}{m}\sum_{\ell=1}^m h(y^j,y^\ell) \widetilde A^\ell_s \bigg) \widetilde{A}^j_s}\ dW^{A,j}_s  \\
& +\int_0^t \left[\sum_{\ell=1}^n \frac{\partial \Phi^{A,n,m}_j}{\partial p^\ell} (\widetilde P_s, \widetilde A_s)\widetilde\eta^{P,\ell}_s + \sum_{\ell=1}^m \frac{\partial \Phi^{A,n,m}_j}{\partial a^\ell} (\widetilde P_s, \widetilde A_s)\widetilde\eta^{A,\ell}_s \right]  ds.
\label{edsTCLbis}
\end{aligned}
\end{equation}
\end{proposition}

Notice that the limiting fluctuation process is an Ornstein-Uhlenbeck process that is centered. Systems of SDEs have already been introduced in the literature to describe the evolution of communities, but they are to our knowledge of a different nature. In \eqref{edsTCLbis} the noise relates to the fluctuation of the stochastic individual-based model around its deterministic limit \eqref{edo2} for $K\rightarrow +\infty$. In other works, such as in \cite{bunin} for instance, the white noise corresponds to a diffusive limit obtained when considering a different longer time-scale, as in the Donsker theorem (see \textit{e.g.} \cite[Section 4.2]{champagnatferrieremeleard}): the random noise comes from the rapid successions of birth and death events in this accelerated time-scale. In other papers, following the steps of May \cite{may72}, \cite{benarousfyodorovkhoruzhenko,fyodorovkhoruzhenko}introduce a system of equations, coupled \textit{via} a smooth random vector field, which describes the approximated dynamical system around an equilibrium state: in this case, the noise models the complexity and nonlinearity of interactions.

\begin{proof}[Proof of Proposition \ref{prop:tcl}]
It is a direct application of Theorem 2.3 from Chapter 11 of \cite{ethierkurtz}. {The proof can also be carried from the semi-martingale expressions of $P^{K,i}_t$ and $A^{K,j}_t$ \eqref{eq:semimartP}-\eqref{eq:semimartA} and \eqref{edo2}. To understand \eqref{edsTCLbis}, consider for example the $i$th component $\eta^{K,P,i}_t$ of $\eta^{K,P}_t$. The other terms can be treated similarly. We sketch here this alternative proof:
\begin{align}
    \eta^{K,P,i}_t= & \eta^{K,P,i}_0+  \sqrt{K}\int_0^t \frac{1}{n}\sum_{\ell=1}^n k(x^\ell,x^i)\Big(\widetilde{P}^\ell_s \widetilde{P}^i_s-\frac{P^{K,\ell}_s}{K} \frac{P^{K,i}_s}{K}\Big)\ ds+\sqrt{K}\frac{M^{K,i}_t}{K}\nonumber\\
    + & \sqrt{K} \int_0^t \Big[g^P\Big(\sum_{j\sim i}c_{ij}^{n,m} \frac{A^{K,j}_s}{K}\Big)\frac{P^{K,i}_s}{K}-g^P\Big(\sum_{j\sim i}c_{ij}^{n,m} \widetilde{A}^{j}_s\Big)\widetilde{P}^{i}_s\Big] ds\nonumber\\
        = & \eta^{K,P,i}_0 +  \frac{1}{n}\sum_{\ell=1}^n k(x^\ell,x^i)\int_0^t \Big[\eta^{K,P,\ell}_s \frac{P^{K,i}_s}{K}+ \widetilde{P}^{\ell}_s \eta^{K,P,i}_s\Big] ds+ \sqrt{K}\frac{M^{K,i}_t}{K}\label{etape2}\\
    + & \int_0^t \Big[g^P\Big(\sum_{j\sim i}c_{ij}^{n,m} \frac{A^{K,j}_s}{K}\Big) \eta^{K,P,i}_s + \sqrt{K} \widetilde{P}^i_s \Big(g^P\Big(\sum_{j\sim i}c_{ij}^{n,m} \frac{A^{K,j}_s}{K}\Big)-g^P\Big(\sum_{j\sim i}c_{ij}^{n,m} \widetilde{A}^{j}_s\Big)\Big)\Big] ds,\nonumber
\end{align}where $M^{K,i}$ is the martingale appearing in \eqref{eq:semimartP} with quadratic variation \eqref{eq:martingaleP}. Because the birth and death rates are assumed to be of class $\Co^1$, so is $g^P$ and a Taylor expansion can be used for the last term:
\begin{align}
    \sqrt{K}\Big(g^P\Big(\sum_{j\sim i}c_{ij}^{n,m} \frac{A^{K,j}_s}{K}\Big)-g^P\Big(\sum_{j\sim i}c_{ij}^{n,m} \widetilde{A}^{j}_s\Big)\Big) 
    =&
     (g^P)'\Big(\sum_{j\sim i}c_{ij}^{n,m} \widetilde{A}^{j}_s\Big) \times \sum_{j\sim i}c_{ij}^{n,m} \sqrt{K}\Big(\frac{A^{K,j}_s}{K}-\widetilde{A}^j_s\Big)+ \varepsilon_K\nonumber\\
     = &  (g^P)'\Big(\sum_{j\sim i}c_{ij}^{n,m} \widetilde{A}^{j}_s\Big) \times \sum_{j\sim i}c_{ij}^{n,m} \eta^{K,A,j}_s + \varepsilon_K,\label{etape3}
\end{align}where $\varepsilon_K$ is a remainder term. From \eqref{etape2} and \eqref{etape3}, we recognize that:
\[\eta^{K,P,i}_t=\eta^{K,P,i}_0+\sqrt{K}\frac{M^{K,i}_t}{K}+\int_0^t\Big[\sum_{\ell=1}^n \frac{\partial \Phi^{P,n,m}_{i}}{\partial p^\ell} ( \widetilde P_s,  \widetilde A_s)\widetilde\eta^{P,\ell}_s + \sum_{\ell=1}^m \frac{\partial \Phi^{P,n,m}_i}{\partial a^\ell} ( \widetilde P_s,  \widetilde A_s)\widetilde\eta^{A,\ell}_s\Big] ds+\varepsilon_K.\]

Using the Aldous-Rebolledo criterion (\textit{e.g.} \cite{joffemetivier}) it is possible to prove that the distributions of the processes $(\eta^{K,P},\eta^{K,A})_K$ form a tight family with a unique limiting value that solves the SDEs \eqref{edsTCLbis}.}
\end{proof}

\section{Continuous limits when abundances and the number of species are large}\label{sec:graphon}

We now consider that the numbers of plant and pollinator species in the network tend to infinity.  Taking the limit $n,m\rightarrow +\infty$, we obtain equations describing the evolution of a population consisting in a \textit{continuum} of species. \\

Let the traits of plants and pollinators be chosen according to i.i.d. random variables with cumulative distribution functions $F_P$ and $F_A$ respectively. To order species according to their respective traits, we proceed as follows: let $(\tilde{u}_i)_{i\geq 1}$ and $(\tilde{v}_j)_{j\geq 1}$ be two sequences of i.i.d. random variables with uniform distribution in $[0,1]$;  for any $n,m\in \N^2$, let $(u^{i,n})_{i\in \lbrac 1,n\rbrac}$ be the ordered $n^{th}$ first values of $(\tilde{u}_i)_{i\geq 1}$ and  $(v^{j,m})_{j\in \lbrac 1,m\rbrac}$ be the ordered $m^{th}$ first values of $(\tilde{v}_j)_{j\geq 1}$, then for any $i\in \lbrac 1,n\rbrac$ and $j\in \lbrac 1,m\rbrac$, \[x^{i,n}=F_P^{-1}(u^{i,n})\qquad \mbox{ and }\qquad y^{j,m}=F_A^{-1}(v^{j,m}).\]
The indices $n$ and $m$ will be dropped when no confusion is possible.\\

Recall that the plant-pollinator network is defined by its adjacency matrix $(G^{n,m}_{ij})_{(i,j)\in \lbrac 1,n\rbrac\times \lbrac 1,m\rbrac }$ and the weights $c^{n,m}_{ij}$ of the interaction $i\sim j$ whenever $G^{n,m}_{ij}=1$. The random variables $G^{n,m}_{ij}$ are supposed to satisfy Assumption \ref{hyp:Gij} for all $n$ and $m\in \N$: the entries of the adjacency matrix are here independent Bernoulli r.v. with parameters $\phi(x^i,y^j)$ for the term $(i,j)\in \lbrac 1,n\rbrac \times \lbrac 1,m\rbrac$. For the harvesting coefficients $c^{n,m}_{ij}$, the following assumption is made:

\begin{hyp}\label{ass_funcc}
We assume that there exists for any $n,m\in \N^2$ a continuous function $c^{n,m}$ satisfying the condition \eqref{def:Cijnm} of Assumption \ref{hyp:cij}.
%such that, conditionally on $(x^{i},y^{j})_{i,j\in \{1,..,n\}\times \{1,..,m\}}$, for all $i\in \{1,..,n\}$, $j\in \{1,..,m\}$,
%\[
%\E[c^{n,m}_{i,j}]=c^{n,m}(x^{i},y^{j}),
%\]
Additionally, we assume that there exists a Lipschitz continuous function $c\, :\, [0,1]^2\mapsto \R$ to which the sequences of functions $nc^{n,m}(.,.)$ and $mc^{n,m}(.,.)$ converge uniformly, in $L^{\infty}([0,1]^2)$, when $n$ and $m$ tend to $+\infty$.
\end{hyp}
The idea in this assumption is that the function $c(.,.)$ is a ``harvesting function'' underlying the matrix $(c_{ij}^{n,m})_{i\in \lbrac 1,n\rbrac,j\in \lbrac 1,m\rbrac}$ of harvesting coefficients. These assumptions are satisfied by Examples~\ref{ex_eco}.

Notice also that Assumption~\ref{ass_funcc} implies that $n$ and $m$ grow to infinity with a similar speed. To state this idea in specific terms, we assume that there exists a sequence $(\alpha_n)_{n\geq 1}$ such that
\begin{equation}
    m=\alpha_n n \quad \text{and} \quad \lim_{n\to\infty}\alpha_n=1.
\end{equation}
In what follows, we write $n\to\infty$, $m\to\infty$ or $n,m\to\infty$ indistinctly.

\begin{proposition}\label{prop_largenbspecies}
Let us assume Assumptions~\ref{hyp:Gij}, \ref{hyp_functions} and \ref{ass_funcc} hold and
that there exist deterministic continuous bounded densities $\bar p_0$ and $\bar a_0$ such that the following weak convergences hold:
\begin{equation}\label{ass_initcond}
\lim_{n\rightarrow +\infty} \frac{1}{n}\sum_{i=1}^n  \widetilde P^i_0 \delta_{x^i}\stackrel{w}{ =} \bar p_0(x)dx,\quad
\lim_{m\rightarrow +\infty} \frac{1}{m}\sum_{j=1}^m  \widetilde A^j_0 \delta_{y^j}\stackrel{w}{ =} \bar a_0(y)dy \quad \text{a.s.}
\end{equation}
{such that for all $n,m\in \N$, the sequences $(\tilde P^i_0)_{i\in \lbrac 1, n\rbrac}$ and $(\tilde A^j_0)_{j\in \lbrac 1, m\rbrac}$ only depend on $(x^i,y^j)_{(i,j)\in \lbrac 1,n \rbrac \times \lbrac 1,m\rbrac}$.} 
We also assume that there is a constant $C>0$ such that a.s., 
\begin{equation}\label{hyp:boundedCI}
\sup_{n,m} \big \langle \widetilde\bP^{n,m}_0,1\rangle + \langle \widetilde\A^{n,m}_0,1\rangle\big)<C.
\end{equation}
Then, for any $T\geq 0$, and for $n,m\rightarrow +\infty$, the sequence of measure-valued processes
  \[\left(\widetilde\bP^{n,m}_t(dx),\widetilde\A_t^{n,m}(dy)\right)_{t\geq 0}:=\Big(\frac{1}{n}\sum_{i=1}^n  \widetilde P^i_t \delta_{x^i},\ \frac{1}{m}\sum_{j=1}^m  \widetilde A^j_t \delta_{y^j}\Big)_{t\geq 0}, \quad n,m\geq 1\]converge in law in $\Co([0,T],\mathcal{M}_F([0,1])^2)$ to a deterministic process $(\bar{\bP},\bar{\A})$ in $\Co([0,T],\mathcal{M}_F([0,1])^2)$, such that:\\
(i) for all $t\geq 0$, $\bar{\bP}_t$ and $\bar{\A}_t$ admit densities $\bar{p}_t$ and $\bar{a}_t$ with respect to the Lebesgue measure on $[0,1]$, \\
(ii) $(\bar{\bP},\bar{\A})$ is the unique solution, for $f\in \Co([0,1],\R)$, of
  \begin{equation} \label{edp:cinet}
      \begin{aligned}
  \int_0^1 f(x)d\bar{\bP}_t(x)= &   \int_0^1 f(x)\bar{p}_0(x)dx\\
  &+\int_0^t \int_0^1 f(x)\Big[ g^P\Big(\int_0^1 c(x,y)\phi(x,y)\bar{a}_s(y)dy\Big)-k\star\bar{p}_s(x)\Big] \bar{p}_s(x)dx\ ds , \\
  \int_0^1 f(y)d\bar{\A}_t(y)= &   \int_0^1 f(y)\bar{a}_0(y)dy\\
    & +\int_0^t \int_0^1 f(y) \Big[g^A\Big(\int_0^1 c(x,y)\phi(x,y)\bar{p}_s(x)dx\Big)-h\star\bar{a}_s(y) \Big]\bar{a}_s(y)dy\ ds,  \end{aligned}
    \end{equation}
    where $k\star\nu$ denotes the convolution product and $\phi$ was introduced in Assumption~\ref{hyp:Gij}. 
In all this statement, the space $\mathcal{M}_F^2([0,1])$ is endowed with its weak topology.
\end{proposition}
The proof of this Proposition is given in Appendix~\ref{app_proof}. Notice that assuming \eqref{hyp:boundedCI} is not very restrictive in our biological context, since by \eqref{ass_initcond},  $\langle \widetilde\bP^{n,m}_0,1\rangle$ (resp. $\langle \widetilde\A^{n,m}_0,1\rangle$) converges a.s. to $\|\bar{p}_0\|_1$ (resp. $\|\bar{a}_0\|_1$) which is finite. This just ensures that the constant $C$ in \eqref{hyp:boundedCI} is not a r.v.\\

Equation \eqref{edp:cinet} is analogous to the ODE system given in Equation \eqref{edo}. However, here species are not considered as discrete but they are continuously distributed along a continuous trait. Modelling the dynamics of a plant-pollinator community with the PDEs given in Equation \eqref{edp:cinet} is thus a functional rather than a species representation of the same system. The connections are modelled in \eqref{edp:cinet} by $\phi$. This function $\phi\, :\, [0,1]^2 \mapsto [0,1]$ is a graphon (see \cite{lovaszbook}): it can be understood (in this case) as a graph on the node sets
$[0, 1]$ for the plants and $[0,1]$ for the pollinators, where $\phi(x,y)$ describes the density of connections between plants $x$ and pollinators $y$.  The term $c \phi$ reflects both the topology ($\phi$) and the intensity ($c$) of plant-pollinator interactions throughout the community depending on the trait values $x$ and $y$ involved.

\begin{remark}
Note that one can formally write a strong form of \eqref{edp:cinet}. This leads to the following system of integral equations on the densities  $\bar{p}_t$ and $\bar{a}_t$:
\begin{equation}
    \label{edp:cinetStrong}
    \begin{aligned}
    \partial_t \bar{p}_t(x)= 
  & \Big[ g^P\Big(\int_0^1 {c(x,y)}\phi(x,y)\bar{a}_t(y)dy\Big) -  k\star\bar{p}_t(x) \Big] \bar{p}_t(x),\\
  \partial_t\bar{a}_t(y)=&  \Big[g^A\Big(\int_0^1 {c(x,y)}\phi(x,y)\bar{p}_t(x)dx\Big)
    -  h\star\bar{a}_t(y) \Big]\bar{a}_t(y) ,
    \end{aligned}
\end{equation}
    with initial conditions $\bar p_0$ and $\bar a_0$.
\end{remark}

\section{Study of the limiting dynamical systems}\label{sec:applications}

Let us now study the behaviour of the limiting dynamical systems that have been obtained: the ODE system \eqref{edo} when $K\rightarrow +\infty$, and the kinetic PDEs \eqref{edp:cinet} when additionally $n,m\rightarrow +\infty$.

\subsection{Stationary states of the ODE system \eqref{edo} }

In this section, we give some results about the dynamics of solutions to Systems~\eqref{edo}, for general forms of $g^P$ and $g^A$.
%The system can be rewritten for all $i\in \{1,..,n\}$ and $j\in\{1,..,m\}$ as
%\begin{equation}\label{eq_sysEDO2}
%    \begin{aligned}
%    \frac{d\widetilde{P}^i_t}{dt} &= \left(g^P\big(\sum_{j\sim i}c^{n,m}_{i,j} \widetilde{A}^j_t\big)-\frac{1}{n}\sum_{\ell=1}^n k(x^i,x^\ell)\widetilde{P}^\ell_t\right)\widetilde{P}^i_t,\\
%    \frac{d\widetilde{A}^j_t}{dt} &= \left(g^A\big(\sum_{i\sim j}c^{n,m}_{i,j} \widetilde{P}^i_t\big)-\frac{1}{m}\sum_{\ell=1}^m h(y^j,y^\ell)\widetilde{A}^\ell_t\right)\widetilde{A}^j_t.    \end{aligned}
%\end{equation}
Notice that any stationary points $(\widetilde{P}^1_\infty,\dots \widetilde{P}^n_\infty, \widetilde{A}^1_\infty,\dots \widetilde{A}^m_\infty)$ of \eqref{edo} are solutions to the following system:
\begin{equation}\label{pointsstationnaires}
\begin{aligned}
\forall i\in \lbrac 1,n\rbrac,\ & \widetilde{P}^i_\infty= 0\mbox{ or }g^P\big(\sum_{j\sim i}c_{ij}^{n,m}\widetilde{A}^j_\infty\big)=\frac{1}{n}\sum_{\ell=1}^n k(x^i,x^\ell)\widetilde{P}^\ell_\infty,\\
\forall j\in \lbrac 1,m\rbrac,\ & \widetilde{A}^j_\infty= 0\mbox{ or }g^A\big(\sum_{i\sim j}c_{ij}^{n,m}\widetilde{P}^i_\infty\big)=\frac{1}{m}\sum_{\ell=1}^m h(y^j,y^\ell)\widetilde{A}^\ell_\infty.
\end{aligned}
\end{equation}

The community with no individual is an obvious stationary point called the \textit{null equilibrium} in the rest of the paper:
$$
\widetilde{P}^i_\infty=0 \text{ for all } i\in \{1,..n\} \quad \text{and}\quad \widetilde{A}^j_\infty=0 \text{ for all } j\in \{1,..m\},
$$
is a stationary state of \eqref{edo}. The local stability of this stationary state depends only on the sign of $g^P(0)$ and $g^A(0)$.

\begin{remark}\label{rque:caspart}
    According to the expression of the non-trivial equilibria, a plant-pollinator community can exist, \textit{i.e.} $\widetilde P^i_\infty>0$ and $\widetilde A^j_\infty>0$, as soon as the interactions between plants and pollinators translate into positive growth rate for both plants and pollinators. Despite that this condition is little restrictive, a non-trivial equilibrium is possible only if trajectories have to converge to it, (\textit{i.e.} a non-trivial equilibrium has to be stable). The stability of an equilibrium can be more tricky to study and obtain in general.
\end{remark}

\begin{lemma}\label{lem_stability}
If $g^P(0)\vee g^A(0)<0$, the null equilibrium 
is locally stable.\\
If $g^P(0)\vee g^A(0) >0$, the null equilibrium is locally unstable.
\end{lemma}

\begin{proof}
From the Hartman–Grobman theorem, we know that local stability is given by the signs of the Jacobian matrix eigenvalues. The Jacobian matrix around the null equilibrium can be directly computed:
\begin{equation*}
    \begin{pmatrix}
    g^P(0)&0&0&...&0&0\\
    0&g^P(0) &0&...&0&0\\
    ...&...&...&...&...&...\\
    0&0&...&0&g^A(0)&0\\
    0&0&...&0&0&g^A(0)\\
    \end{pmatrix},
\end{equation*}
whose eigenvalues are $g^P(0)$ with multiplicity $n$ and $g^A(0)$ with multiplicity $m$. This ends the proof.
\end{proof}

\subsection{Study of a plant-pollinator interaction with a trade-off for plants}

As explained in Remark \ref{rque:caspart}, more precise functional forms for the birth and death rates are needed to carry out further computation.\\

In this Section, we give results for the following particular forms of the individual growth rate functions:
\begin{equation}
\begin{aligned}
g^P(R^A)=& b^P(R^A)-d^P(R^A)=\frac{\alpha_P R^A}{\beta_P+\gamma_P R^A}-(d_P+\delta_P R^A), \\
g^A(R^P)=& b^A(R^P)-d^A(R^P)=\frac{\alpha_A R^P}{\beta_A+\gamma_A R^P}-d_A,\label{growth-rates}
\end{aligned}
\end{equation} 
where $R^A$ and $R^P$ are the total resources collected respectively by plants and pollinators (Eqs. \eqref{Rij:model2} and \eqref{Rij:model3}). The parameters $\alpha_P$, $\alpha_A$, $\beta_P$, $\beta_A$, $\gamma_P$, $\gamma_A$, $d_A$, $d_P$ and $\delta_P$ are assumed positive. Note that according to Lemma~\ref{lem_stability}, the null equilibrium is stable in this particular case.

On the right-hand side of Equation~\eqref{growth-rates}, the death rate is an increasing function of the resources exchanged for plants $R^A$. %, but a decreasing function for pollinators.
This reflects an interaction trade-off for the plant, \textit{i.e.} it is supposed that there is a cost for interacting with pollinators due to nectar production, leaves consumption, etc. %On contrary, it is supposed that pollinators always increase their benefits when interacting with plants. \\
Graphical representations of $g^P$ and $g^A$ are given in Figure~\ref{fig_gAgP}.
\begin{figure}[h]
\begin{center}
    \includegraphics[width=1.0\textwidth]{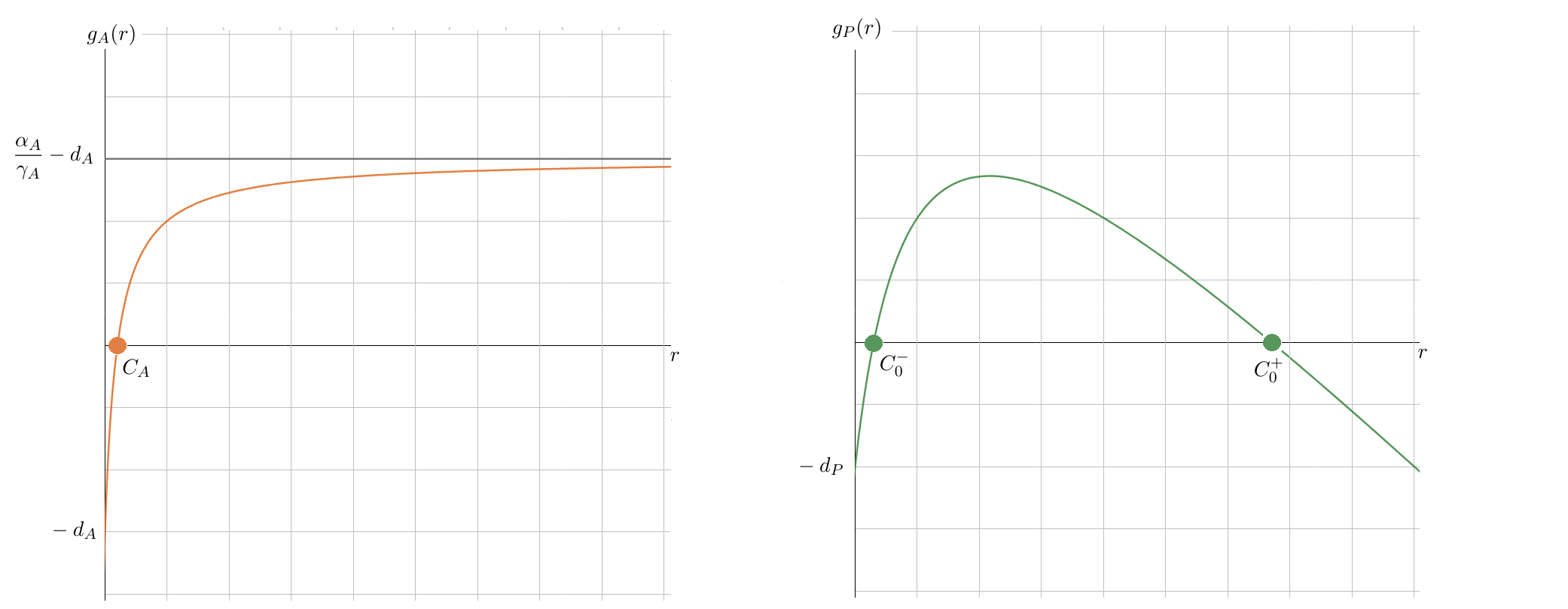}
    \caption{\label{fig_gAgP} Graphical representation of $g_A$ and $g_P$}
\end{center}
\end{figure}
The form given here to the growth rates (or \textit{numerical responses}) $g^P$ and $g^A$ can be found in \cite{holland2008}. It assumes that there is no competition between plants for collecting resources from pollinators, and that there is no competition between pollinators for collecting resources from plants.\\

We state an additional assumption on the parameters of both functions to avoid the case where one growth rate is never positive.
\begin{hyp}\label{hyp:CNequilibre}
Assume that there exist $r^P>0$ and $r^A>0$ such that
$$
g^P(r^P)>0 \quad \text{and} \quad g^A(r^A)>0.
$$
\end{hyp}

\subsubsection{Case $n=m=1$} 

Let us first consider the particular case of a single plant species interacting with a single pollinator species. In other words, we consider the system \eqref{edo} with $n=m=1$ and omit the indices and exponents $i=1$ and $j=1$ for the sake of simplicity (so that $\widetilde{P}^1_t$ becomes ${P}_t$ and $c_{11}^{1,1}$ becomes $c$ for instance). We also assume that the two species interact, \textit{i.e.} $G_{11}=1$. Otherwise, computations are trivial: the community goes to extinction as only death events occur. We have then:
\begin{equation} 
\begin{aligned}
    \frac{d{P}_t}{dt} &= \left(g^P\big(c {A}_t\big)-k{P}_t\right){P}_t \\
    \frac{d{A}_t}{dt} &= \left(g^A\big(c {P}_t\big)-h{A}_t\right){A}_t \label{eq_1species}
\end{aligned}
\end{equation}
(see \cite{holland2010} for a similar ODE model in the ecology literature).

As discussed in Lemma \ref{lem_stability}, the null equilibrium is a locally stable equilibrium of the system \eqref{eq_1species}. Let us search for positive equilibrium. To help the study, we may draw some phase plan of the system.
(see Figure \ref{fig:phaseplan}). 

\begin{figure}[!ht]
    \centering
    \begin{minipage}{0.48\textwidth}
    \includegraphics[width=0.95\textwidth]{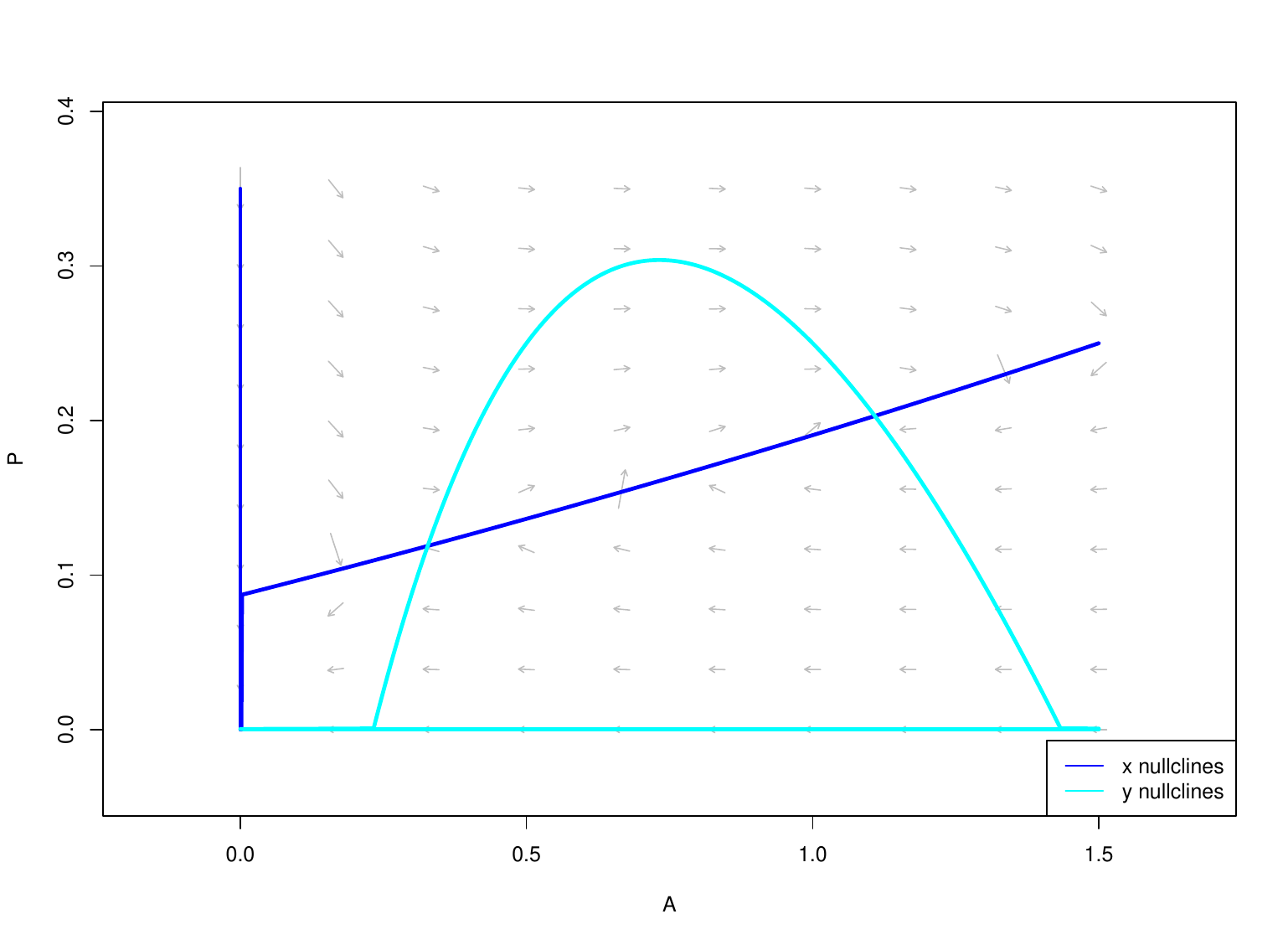}
    \end{minipage}
    \begin{minipage}{0.48\textwidth}
    \includegraphics[width=0.95\textwidth]{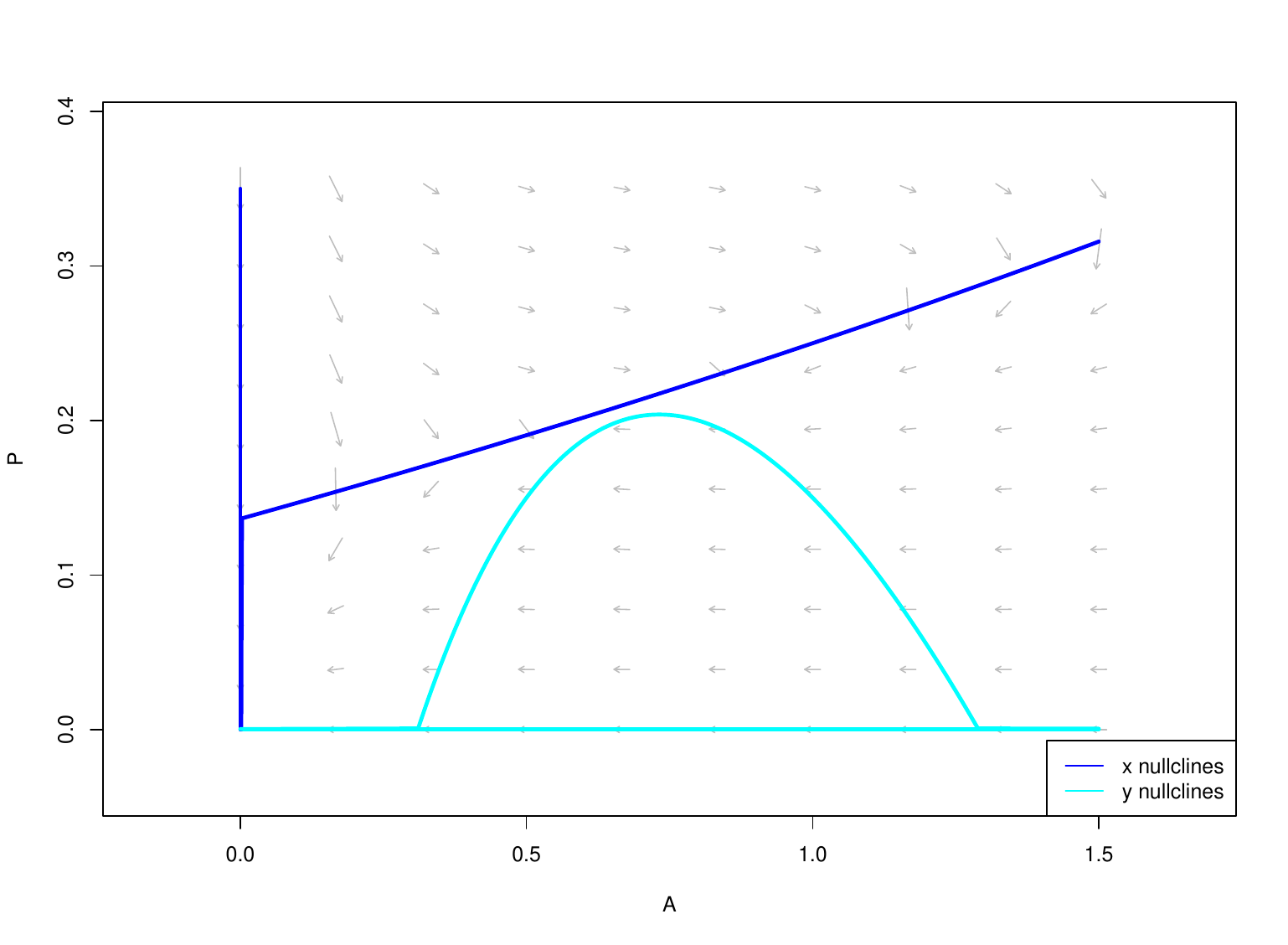}
    \end{minipage}
    \caption{Phase plan and nullclines of the system of ODE \eqref{eq_1species}: nullclines for the pollinators dynamics in blue; nullclines for the plants dynamics in cyan. For the plots, the parameters are $\alpha_A=25$, $\alpha_P=9$, $\beta_A=\beta_P=\gamma_A=\gamma_P=1$ and $\delta_P=3$, and some various $d_A$ and $d_P$. Right: $d_A=2$ and $d_P=1$, the system has $3$ stationary states: the null equilibrium, $1$ stable positive equilibrium and $1$ unstable positive equilibrium. Left: $d_A=3$ and $d_P=1.2$, the unique equilibrium of the system is the null equilibrium. }
    \label{fig:phaseplan}
\end{figure}

Precisely, we prove the following lemma that gives the number of equilibria and their stability. 

\begin{lemma}\label{lem_equicp}
In addition to the null-equilibrium, System~\eqref{eq_1species} has $0$, $1$ or $2$ equilibria, depending on the values of the parameters.\\
Precisely, denote by $C_0^-$ and $C_0^+$ the two zeros of $g^P$. System~\eqref{eq_1species} has $2$ positive equilibrium if and only if $$\max_{x\in (C^-_0/c,C^+_0/c)}f(x)>0.$$ 
In this case, the equilibria are $({P}^-_\infty,{A}^-_\infty)$ and $({P}^+_\infty,{A}^+_\infty)$ with ${A}^-_\infty<{A}^+_\infty$ and they are the two positive solutions to
\begin{equation}\label{eq_equilibria}
\left\{
\begin{aligned}
&{P}_\infty=\frac{1}{k}g^P(c{A}_\infty)\\
&{A}_\infty=\frac{1}{h}g^A\left(\frac{c}{k}g^P(c{A}_\infty)\right), \text{ and } c{A}_\infty\in (C_0^-,C_0^+).
\end{aligned}
\right.
\end{equation}
$({P}^-_\infty,{A}^-_\infty)$ is an unstable positive equilibrium and $({P}^+_\infty,{A}^+_\infty)$ is a stable positive equilibrium.
\end{lemma}
To conclude on the dynamics of the trajectories of System~\eqref{eq_1species}, we performed simulations. 
In summary, it appeared that, depending on the parameters values, either all trajectories are attracted by the null equilibrium (Fig~\ref{fig:phaseplan}(left)), or trajectories converge to $0$ or some positive equilibrium, depending on the initial conditions (Fig~\ref{fig:phaseplan}(right)). We never observed cycles. Moreover, the competitive terms ensure that trajectories remain bounded.

\begin{proof}[Proof of Lemma \ref{lem_equicp}]
Any stationary point $({P}_\infty,{A}_\infty)$ is solution to 
\begin{equation}
\begin{aligned}
    g^P\big(c {A}_\infty\big)=k{P}_\infty & \quad \text{ or }\quad {P}_\infty=0 \\
    g^A\big(c {P}_\infty\big)=h{A}_\infty &\quad\text{ or }\quad {A}_\infty=0. \label{eq_ss1species}
\end{aligned}
\end{equation}
Recall that either both or no coordinates are null. 
In view of \eqref{eq_ss1species}, the existence of a positive equilibrium requires at least that $\sup_{x\in\R^+} g^P(x)>0$ and $\sup_{y\in \R^+}g^A(y)>0$. This is true under Assumption \ref{hyp:CNequilibre}. Recall that $C_0^-$ and $C_0^+$ are the two zeros of $g^P$. From~\eqref{eq_ss1species}, we deduce that, if there exists a positive stationary state, it satisfies System~\eqref{eq_equilibria}.

It remains to find the number of zeros of $f(x)=\frac{1}{h}g^A\left(\frac{c}{k}g^P(cx)\right)-x$ on $\left(\frac{C^-_0}{c},\frac{C^+_0}{c}\right)$. Since $g^A(c/k \times \cdot)$ and $g^P(c \times \cdot)$ are concave functions with $g^A$ non-decreasing, their composition is also concave and so is $f$. Since $f(C_0^\pm/c)=-d_A/h-C_0^\pm/c<0$, it admits
\begin{itemize}
    \item $2$ zeros if $\max_{x\in (C^-_0/c,C^+_0/c)}f(x)>0$, and System~\eqref{eq_ss1species} has $2$ positive equilibria,
    \item $1$ zero if $\max_{x\in (C^-_0/c,C^+_0/c)}f(x)=0$, and System~\eqref{eq_ss1species} has $1$ positive equilibrium,
    \item no zero otherwise, and System~\eqref{eq_ss1species} has no positive equilibrium.
\end{itemize}
Finally, in the case of $2$ positive equilibria, which will be denoted by $({P}^-_\infty,{A}^-_\infty)$ and $({P}^+_\infty,{A}^+_\infty)$ with ${A}^-_\infty<{A}^+_\infty$, the Jacobian matrix can be computed in order to study their stability (see the Hartman–Grobman theorem):
$$
J^\pm=\begin{pmatrix}
-k {P}^\pm_{\infty} & c{P}^\pm_{\infty}(g^P)'(c{A}^\pm_{\infty})\\
c{A}^\pm_{\infty}(g^A)'(c{P}^\pm_{\infty}) & -h {A}^\pm_{\infty}
\end{pmatrix}.
$$
The trace is negative, thus the stability of the stationary states depend on the sign of the determinant.
$$
\det(J^\pm)={P}^\pm_{\infty}{A}^\pm_{\infty}\Big(hk-c^2(g^P)'(c{A}^\pm_{\infty})(g^A)'(c{P}^\pm_{\infty})\Big)=-{P}^\pm_{\infty}{A}^\pm_{\infty}hkf'({A}^\pm_{\infty}),
$$
and its sign depends only on the sign of $f'({A}^\pm_{\infty})$.
According to the previous study of the function $f$, we deduce that $\det(J^+)$ is negative, and $\det(J^-)$ is positive. In other words, System~\eqref{eq_ss1species} admits $1$ stable positive equilibrium and $1$ unstable positive equilibrium.\\
In the case when System~\eqref{eq_ss1species} admits a unique positive equilibrium, the same study implies that this equilibrium is a non-hyperbolic equilibrium.
\end{proof}

\subsubsection{Behaviour of the kinetic equations \eqref{edp:cinet}}

%In general, as in this particular case with function \eqref{growth-rates}, the study of the stationary states of System~\eqref{eq_sysEDO2} is difficult, because it implies the study of scalar systems and matrices in large dimension. Results can be done with computation for particular values of the parameters, but theoretical results are out of reach. Studying the stationary states of the limiting continuous system~\eqref{edp:cinet} can be more convenient, which gives all the interest of this limiting equation.\\

Recall that we are still working with the growth rates defined in \eqref{growth-rates}. \\

The stationary solutions of \eqref{edp:cinet} are couples of measures $ \bar{P}^\infty(dx)$ and $ \bar{A}^\infty(dy)$ in $\mathcal{M}_{F}([0,1])$ such that, for any positive, bounded and continuous function $f$ on $[0,1]$:
  \begin{align}
0= &\int_0^1f(x)  \Big[ g^P\Big(\int_0^1 \psi(x,y)\bar{A}^\infty(dy)\Big) - k\star \bar{P}^\infty(x)\Big] \bar{P}^\infty(dx), \label{sol:stationnaire-edp1} \\
0= &  \int_0^1f(y) \Big[g^A\Big(\int_0^1 \psi(x,y)\bar{P}^\infty(dx)\Big)-h\star\bar{A}^\infty(y) \Big]\bar{A}^\infty(dy)\label{sol:stationnaire-edp2}
\end{align}
where \begin{equation}\label{def:psi}
\psi(x,y):=c(x,y)\phi(x,y).\end{equation}

The null measures constitute a trivial solution to \eqref{sol:stationnaire-edp1}-\eqref{sol:stationnaire-edp2}, which is stable in our particular case~\eqref{growth-rates}. Let us discuss non-zero solutions.

\begin{proposition}\label{prop_stst}
Assume that 
\begin{itemize}
    \item the competitive kernels $k$ and $h$ are constant functions;
    \item for all $x_0,y_0\in [0,1]$, $y\mapsto \psi(x_0,y)$ and $x\mapsto \psi(x,y_0)$ are increasing and continuous functions;
\end{itemize} 
then, System~\eqref{edp:cinet} does not admit any non-null stationary state with densities w.r.t Lebesgue measure.\\
Moreover, any non-null stationary state in $L^1([0,1]^2)$ is a couple of measures $(\bar P^\infty,\bar A^\infty)$ such that
\begin{equation}\label{form_stst}\begin{aligned}
\exists \bar{a}_0,& \bar{p}_1\in\R^*_+, \bar{p}_2\in \R_+, \bar x_1,\bar x_2,\bar y_0\in [0,1],\  \left\{\begin{aligned}&\bar{P}^\infty=\bar{p}_1\delta_{\bar x_1}+\bar{p}_2\delta_{\bar x_2}\\
&\bar{A}^\infty=\bar{a}_0\delta_{\bar y_0}
\end{aligned}\right. \\
&\text{ with } \left\{ \begin{aligned} &g^P\Big(\bar{a}_0 \psi(\bar x_1,\bar y_0)\Big)=g^P\Big(\bar{a}_0 \psi(\bar x_2,\bar y_0)\Big)=k(\bar{p}_1+\bar{p}_2)\\
&g^A\Big(\bar{p}_1 \psi(\bar x_1,\bar y_0)+\bar{p}_2 \psi(\bar x_2,\bar y_0)\Big)=h\bar{a}_0.
\end{aligned}\right.
\end{aligned}
\end{equation}
All these stationary states are unstable, except the state
\begin{equation}\label{stablest}
\left\{\begin{aligned}&\bar{P}^\infty=\frac{\max_{\R^+}g^P}{k}\ \delta_{x_0}\\
&\bar{A}^\infty=\frac{ \arg\max_{\R^+}g^P}{\psi(x_0,1)}\ \delta_{1}
\end{aligned}\right. \\
\end{equation}
if $x_0$, solution to
\begin{equation}\label{hyp_stable}
g^A\left(\frac{\max_{\R^+}g^P}{k}\psi(x_0,1)\right)\psi(x_0,1)=h\cdot \arg\max_{\R^+}g^P,
\end{equation}
exists and is unique.\\
Finally, assuming that, for all initial conditions with positive densities w.r.t Lebesgue measure, the quantities $\int_0^1\psi(x,.)\bar p_t(x)dx$, $\int_0^1\bar p_t(x)dx$, $\int_0^1\psi(.,y)\bar a_t(y)dy$ and  $\int_0^1\bar a_t(y)dy$ converge when $t$ grow to infinity, then the trajectory converges to equilibrium \eqref{stablest}.
\end{proposition}

In the proof of Proposition~\ref{prop_stst}, we will see that the result is not restricted to the specified forms of $g^P$ and $g^A$ given in \eqref{growth-rates}. The proposition is still true as soon as shapes (successions of increases and decreases) of $g^P$ and $g^A$ are the same as the ones of the functions specified in \eqref{growth-rates}. \\
We will also see that in general, System~\eqref{edp:cinet} has no stationary state with densities, all stationary states will be composed of Dirac measures. Moreover, the maximal number of Dirac measures found in such a stationary state corresponds to the maximal number of points that can be found in the inverse image of a positive real number for the functions $g^P$ and $g^A$. The stability of these stationary states can then be deduced using the same kind of arguments as the ones in the following proof.\\

Proposition \ref{prop_stst} shows that the only possible stable equilibrium is composed of only one plant species and one pollinator species. Simulations below will illustrate this Proposition. %It means that that the number of pollinator species and plant species is reduced to $1$ when considering all types of growths given in Figure 2 of \cite{holland2008mutualism}. \\
This result shows that when the plant-pollinator network is nested and the competition among plants and among pollinators is constant, then the plant-pollinator community collapses to a single plant-pollinator species pair. At the stationary state, the value of the trait for the pollinators is 1, where the probability of interactions with plant is maximized. The trait for plants is <1 \eqref{hyp_stable}, because of the trade-offs for plants due to the cost of the interactions. This is in line with the numerical analysis of a system of ODEs by \cite{lever2014sudden}. Our results is a formal demonstration of this necessary collapse. It is also more general since we show that it does not depend on the specific form of $g^P$ and $g^A$. This raises the question of whether it is possible, in this type of ecological model, to maintain a stable coexistence of several plant and pollinator species in a single community by modifying the structure of the interaction graph (\textit{i.e.} with other assumptions for the function $\psi$), or the structure of the competition graph (\textit{i.e.} with non-constant competition functions $k$ and $h$).

\begin{proof}
Any stationary state $(\bar P^\infty,\bar A^\infty)$ in $L^1([0,1]^2)$ satisfies that 
\begin{equation}
    \left\{
    \begin{aligned}
    g^P\Big(\int_0^1 \psi(x,y)\bar{A}^\infty(dy)\Big) = \int_0^1 k \bar{P}^\infty(dx'), \text{ for all } x\in \text{supp}\bar P^\infty, \\
    g^A\Big(\int_0^1 \psi(x,y)\bar{P}^\infty(dx)\Big) = \int_0^1 h \bar{A}^\infty(dy'), \text{ for all } y\in \text{supp}\bar A^\infty
    \end{aligned}
    \right.
\end{equation}
Since $(\bar P^\infty,\bar A^\infty)\in L^1([0,1]^2)$, $\int_0^1 k \bar{P}^\infty(dx')$ and $\int_0^1 h \bar{A}^\infty(dy')$ are finite. Then, according to the continuity and the variations of the functions ($\psi$ is increasing w.r.t each variable, $g^A$ is increasing and $g^P$ is increasing then decreasing), we deduce \eqref{form_stst}.\\
Let us denote by $\mathcal{S}$ the set of stationary states given by \eqref{form_stst} excluding equilibrium~\eqref{stablest}.
Let us prove that all states in $\mathcal{S}$ are unstable. To this aim, we develop an argument by contradiction, assuming that it is not true and that there exists at least one stationary state $(\bar P^\infty,\bar A^\infty)\in \mathcal{S}$. We deal with the case where $\bar x_1<x_0<\bar x_2$, $\bar y_0$ being any value in $[0,1]$. Other cases, (2) $\bar x_2<x_0<\bar x_1$, $\bar y_0\in [0,1]$ or (3) $\bar x_1=\bar x_2=x_0$, $\bar y_0<1$, can be treated similarly.\\
Let us consider an initial state that is close (in Wasserstein distance) to $(\bar P^\infty,\bar A^\infty)$ and that has positive densities w.r.t Lebesgue measures.
Thus, $p_t$ exists and is positive for all $t\geq 0$ and we can study $\log(\bar p_t(x))$ for all $t\in \R_+$ and all $x\in [0,1]$:
$$
\frac{d}{dt}\log(\bar p_t(x))= g^P\Big(\int_0^1 \psi(x,y)\bar a_t(y)dy\Big) - \int_0^1 k , \bar p_0 \,t(x')dx'.
$$
Under our assumption, $\mathcal{S}$ admits at least a stable state and since the initial state is close to this set, we deduce that $(\bar p_t,\bar a_t)$ will converge to some $(\hat P^\infty,\hat A^\infty)\in \mathcal{S}$, which is close to $(\bar P^\infty,\bar A^\infty)$ (if not it), sufficiently close to have a similar form. Then for all $x\in [0,1]$,
$$
\begin{aligned}
\frac{d}{dt}\log(\bar p_t(x))
\underset{t\to\infty}{\longrightarrow}& g^P\Big(\int_0^1 \psi(x,y)\hat a^\infty(y)dy\Big) - \int_0^1 k\hat p^\infty(x')dx'\\
&= g^P\Big(\int_0^1 \psi(x,y)\hat a^\infty(y)dy\Big)-g^P\Big(\int_0^1 \psi(\hat x_1,y)\hat a^\infty(y)dy\Big).
\end{aligned}
$$
The latter quantity is positive as soon as $x\in (\hat x_1,x_0)$, which contradicts the fact that $\bar p_t(x)$ converges to $0$ for all $x\not\in \{\hat x_1,\hat x_2\}$.\\
The last point of Proposition~\eqref{prop_stst} can be obtained using the same argument, once the convergences of the four quantities detailed in the statement are assumed.
\end{proof}

Let us illustrate numerically Proposition~\ref{prop_stst} (see Figure \ref{fig:equilibriaAP1000}). To do so, we have discretized the continuous model \eqref{edp:cinetStrong}. Let $N$ be a positive integer and let $(p_i(t))_i$, resp. $(a_j(t))_j$ be approximations of $(\bar p_t(x_i))_i$, resp. $(\bar a(y_j))_j$, at $x_i= i/N $ and $y_j= j/N $, $0\leq i,j \leq N$. The solution of \eqref{edp:cinetStrong} can be approximated using the rectangular rule by the following system of coupled ODEs:
\begin{equation}\label{EDPapprox}
    \begin{aligned}
    \frac{d p_i}{d t} = \left[ g^P \left ( \frac{1}{N} \sum_{j=0}^N c_{ij} a_j \right) - \frac{1}{N}\sum_{j=0}^N k_{i,j} p_j   \right] p_i, \\
    \frac{d a_j}{d t} = \left[ g^A \left ( \frac{1}{N} \sum_{i=0}^N c_{ij} p_i \right) - \frac{1}{N}\sum_{i=0}^N h_{i,j} a_i   \right] a_j,
    \end{aligned}
\end{equation}
where the initial data has been defined as $p_i(0) = \bar{p}_0(x_i)$, $p_i(0) = \bar{p}_0(x_i)$,  $a_j(0) = \bar{a}_0(x_j)$. We recognize the ODE system \eqref{edo} obtained when the numbers of plant and pollinator species are finite with $n=m=N$.

We shall consider the case of growth functions \eqref{growth-rates}. In order to fit the hypotheses of Proposition \ref{prop_stst} and ensure the convergence towards an explicit equilibrium of the dynamics, the interaction matrix $(c_{ij})_{ij}$ is non-decreasing in each index $i$ and $j$, and the competition kernels will be chosen constant. We consider the simple case \[c_{ij}=\frac{(i+1)(j+1)}{2N^2}.\]

We represent in Figure \ref{fig:equilibriaAP1000} (in logarithmic scale) the convergence towards the equilibrium state reached when no evolution occurs anymore on the discrete dynamics, for a random initial datum supported in $[0,1]$. We checked numerically the conclusion of Proposition \ref{prop_stst}: the system converges towards an equilibrium which consists of a Dirac delta centered inside the domain for the plants, and on the right border for the pollinators. The last remaining plant species is moderately specialist ($x \simeq 0.6$) whereas the remaining pollinator is genuinely generalist ($y = 1$).

\begin{figure}
    \centering
    \includegraphics[scale=.32]{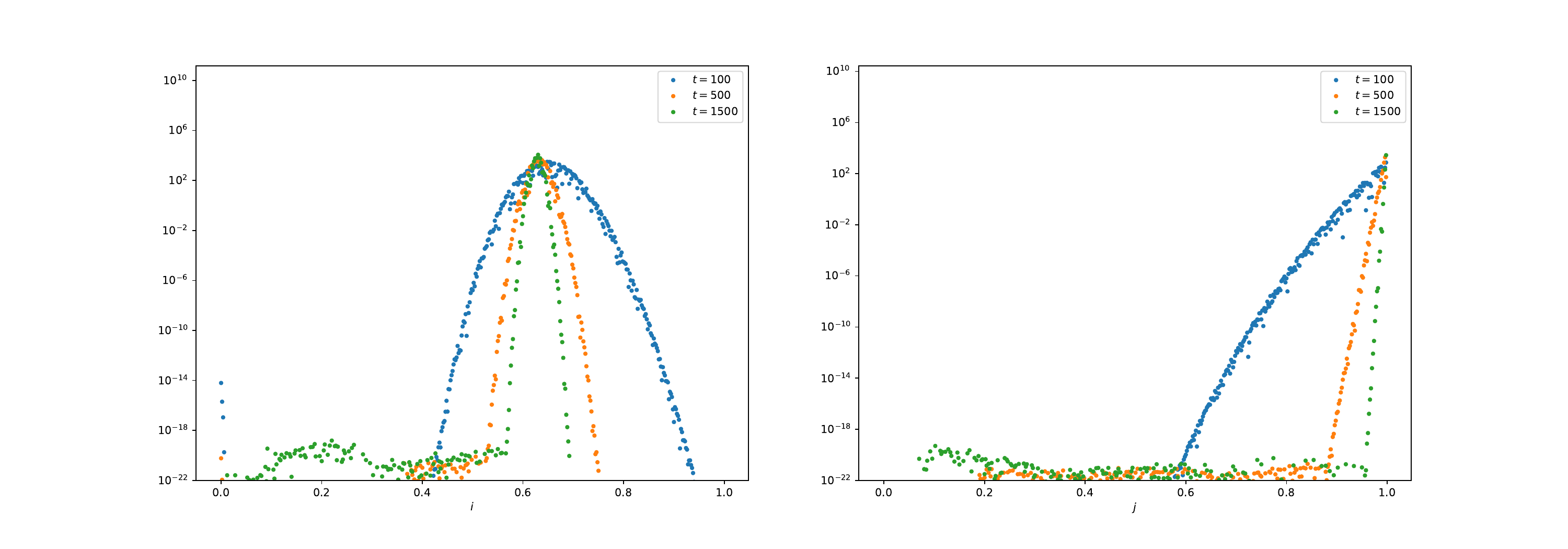}
    \caption{Solutions to \eqref{EDPapprox} at time $100$ (blue), $500$ (orange) and $1500$ (green). The left panel represents the distribution of plant species and the right panel the distribution of pollinator species. Parameters are $\alpha_A=3$, $\alpha_P=25$, $\beta_A=\beta_P=\gamma_P=1$, $\gamma_A=0.3$, $d_P=1$, $d_A=3$ and $\delta_P=3$, with $N=500 $species of plants and pollinators.}
    \label{fig:equilibriaAP1000}
\end{figure}

\section{Conclusion}

Two of the main goals of Ecology are i) to explain how emerged the size, composition and structure of communities, for example plant-pollinator communities, and ii) predict the dynamics and stability of a given community. Such questions involve different hierarchical scales, from the individuals which effectively interact (the microscopic scale), to the species and the whole community (the macroscopic scale). 

Theoretical Ecology mostly addresses these questions by studying systems of ODEs where each equation refers to a given species. This approach has several methodological and conceptual drawbacks. First, the system is discretized and structured by species, which precludes the possibility of within-species variability regarding the rate and intensity of interactions, in particular between-species overlaps. Second, the interaction graphs underlying the system of ODEs are most often arbitrarily given without specified mechanisms. Finally, due to the high-dimensionality of the ODEs system, it is difficult to obtain general properties of communities.

Here, we modelled an individual-based  plant-pollinator network, where interactions are structured by an individual trait. We found continuous limits of the microscopic system and finally showed that it could be approximated by PDEs. We finally studied the dynamics, stationary state and stability of these continuous limits.
Our approach allows to address several limits exposed before: the relationships between the individuals, species and community scales are explicit; interactions variability within and between-species is taken into account; the interaction graph is based on individuals' traits; PDEs approximations allow an explicit and analytical study of the property of a community. We showed in particular that a nested plant-pollinator network is expected to collapse, a phenomenon already observed by previous works, but to our knowledge for the first time formally demonstrated under general conditions. Our approach can thus provide a new and original theoretical framework for ecologists to address long-standing questions. 

One of the actual limit of our model is that  interactions take place through a ``mass-action model", in particular with the use of the resources \eqref{Rij:model2}-\eqref{Rij:model3} which are already at a macroscopic scale. Returning to an event-based modelling for establishing these functional responses (e.g. \cite{bansayebilliardchazottes}) and showing how interactions at the individual level would translate into an interaction graph and into the dynamics of the whole community is an open question.

\section*{Code}

The code for simulations is available here: \\
\verb"https://gitlab.com/thoma.rey/PlantPollinatorsNetwork/"

\appendix

\section{Proofs}

\subsection{Bounds on the microscopic representation}\label{app_proof1}

Here, we give a sketch of proof of Proposition~\ref{prop:stochpro}. The proof follows from usual stochastic calculus with Poisson point processes, as developed in \cite{fourniermeleard} for example.\\

\begin{proof}[Proof of Proposition~\ref{prop:stochpro}] The proof is divided into several steps.\\

\noindent \textbf{Step 1}: let us first prove some moment estimates, including \eqref{propagation:moment}. First, notice that the process can be stochastically bounded by a process $(\tilde{\mathbf{P}}^{K,n}_t,\tilde{\mathbf{A}}^{K,m}_t)$ where only births occur at the maximal birth rate, denoted by $M^P:=\sup_{r\in \R^+}b^P(r)$:
\begin{align*}
    \langle \mathbf{P}^{K,n}_t,\one\rangle &=  \frac{1}{nK}\sum_{i=1}^n P^{K,i}_t 
   \\ & \leq \langle \tilde{\mathbf{P}}^{K,n}_t,\one\rangle = \langle \mathbf{P}^{K,n}_0,\one\rangle +  \int_0^t \int_{E}  \frac{1}{nK} \sum_{i=1}^n\ind_{i=k}\ind_{\theta\leq  M^P\tilde P^{K,i}_{s-}} Q^P_B(ds,dk,d\theta).
\end{align*}
Since
\[
\E\Big(\sup_{s\leq t\wedge \tau^K_N}\langle \mathbf{P}^{K,n}_s,\one\rangle\Big)\leq 1+ \E\Big(\sup_{s\leq t\wedge \tau^K_N}\langle \mathbf{P}^{K,n}_s,\one\rangle^2\Big)\leq 1+ \E\Big(\sup_{s\leq t\wedge \tau^K_N}\langle \tilde{\mathbf{P}}^{K,n}_s,\one\rangle^2\Big),
\]
we only have to find a bound for the second moment of $\langle \tilde{\mathbf{P}}^{K,n}_s,\one\rangle$.\\
For a constant $N>0$, we introduce the stopping time 
\[\tau^K_N=\inf\Big\{ t\geq 0,\ \langle \tilde{\mathbf{P}}^{K,n}_t,\one\rangle\geq N\mbox{ or }\langle \tilde{\A}^{K,m}_t,\one\rangle \geq N \Big\}.\]
Using Itô's formula for jump processes (\textit{e.g.} \cite[Th.5.1 P.66]{ikedawatanabe}), we have:
\begin{align*}
    \langle & \tilde{\mathbf{P}}^{K,n}_t,\one\rangle^2\\ = & \langle \mathbf{P}^{K,n}_0,\one\rangle^2 +  \int_0^t \int_{E} \sum_{i=1}^n\ind_{k=i} \Big(\big(\langle \tilde{\mathbf{P}}^{K,n}_{s-},\one\rangle+\frac{1}{nK}\big)^2 - \langle \tilde{\mathbf{P}}^{K,n}_{s-},\one\rangle^2\Big) \ind_{\theta\leq  M^P \tilde{P}^{K,i}_{s-}} Q^P_B(ds,dk,d\theta)\\
    =&  \langle \mathbf{P}^{K,n}_0,\one\rangle^2 +  \int_0^t \int_{E} \sum_{i=1}^n\ind_{k=i} \Big(\frac{2}{nK} \sup_{u\leq s{-}}\langle \tilde{\mathbf{P}}^{K,n}_{u},\one\rangle + \frac{1}{n^2K^2}\Big) \ind_{\theta\leq  M^P \tilde{P}^{K,i}_{s-}} Q^P_B(ds,dk,d\theta).
\end{align*}
Since the right hand side is an increasing process, we can replace the left hand side by $\sup_{s\leq t \wedge \tau_N^K}\langle \tilde{\mathbf{P}}^{K,n}_s,\one\rangle^2$. Then, taking the expectation gives:
\begin{align*}
    \E\Big(\sup_{s\leq t\wedge \tau^K_N}\langle \tilde{\mathbf{P}}^{K,n}_s,\one\rangle^2\Big)  \leq & \E\big(\langle \mathbf{P}^{K,n}_0,\one\rangle^2\big) + \E\Big[ \int_0^{t\wedge \tau^K_N} \sum_{i=1}^n  \Big(\frac{2}{nK} \sup_{u\leq s}\langle \tilde{\mathbf{P}}^{K,n}_{u},\one\rangle + \frac{1}{n^2K^2}\Big) M^P \tilde{P}^{K,i}_s \ ds \Big]\\
    \leq & \E\big(\langle \mathbf{P}^{K,n}_0,\one\rangle^2\big) + \int_0^t M^P\Big[ (2+ \frac{1}{nK})  \E \Big(1+\sup_{u\leq s\wedge \tau^K_N}\langle \tilde{\mathbf{P}}^{K,n}_{u},\one\rangle^2 \Big) \Big]  \ ds.
    \end{align*}
Using Gronwall's lemma, we deduce
\begin{align*}
     \E\Big(\sup_{s\leq t\wedge \tau^K_N}\langle \tilde{\mathbf{P}}^{K,n}_s,\one\rangle^2\Big) \leq & \Big[\E\big(\langle \mathbf{P}^{K,n}_0,\one\rangle^2\big)+ 3M^P\Big] e^{3M^P t} <+\infty.
    \end{align*}
    Since the bound does not depend on $N$, we make $N$ grow to $\infty$ and obtain
    \begin{align*}
     \E\Big(\sup_{s\leq t}\langle \tilde{\mathbf{P}}^{K,n}_s,\one\rangle^2\Big) \leq & \Big[\E\big(\langle \mathbf{P}^{K,n}_0,\one\rangle^2\big)+ 3M^P\Big] e^{3M^P t} <+\infty.
    \end{align*}
    Similarly, 
    \begin{equation}
           \E\Big(\sup_{s\leq t}\langle \tilde{\mathbf{A}}^{K,m}_s,\one\rangle^2\Big) \leq  \Big[\E\big(\langle \mathbf{A}^{K,m}_0,\one\rangle^2\big)+ 3M^A\Big] e^{3M^A t} <+\infty.
    \end{equation}

Notice that the upper bounds that we obtain do not depend on $K$. 
Once these estimates have been obtained, existence and uniqueness of a strong solution to the stochastic differential equation is given by \cite[Th. IV.9.1]{ikedawatanabe}.\\

\noindent \textbf{Step 2:} Let us now consider a measurable real function $f$ on $[0,1]$. Using It\^o's formula for SDEs with jumps (see \textit{e.g.} \cite[p. 66-67]{ikedawatanabe}) and \eqref{def_Poisson1} provides \eqref{eq:semimartP} in which:
\begin{align}
M^{K,i}_t= &
\int_0^t \int_{E}  \ind_{i=k} \ind_{\theta\leq  b^P\big(\sum_{j\sim i}R^{A,K,ij}_{s-}\big)P^{K,i}_{s-}} Q^P_B(ds,dk,d\theta)- \int_0^t b^P\big(\sum_{j\sim i}R^{K,ij}_{s}\big)P^{K,i}_{s} ds \nonumber\\
  - & \int_0^t \int_{E}  \ind_{i=k} \ind_{\theta\leq \big[d^P\big(\sum_{j\sim i}R^{A,K,ij}_{s-}\big)-k\star \mathbf{P}^{K,n}_{s-}(x^i)\big]P^{K,i}_{s-}} Q^P_D(ds,dk,d\theta) \nonumber\\
  + &  \int_0^t \big[d^P\big(\sum_{j\sim i}R^{A,K,ij}_{s}\big)-k\star \mathbf{P}^{K,n}_s(x^i)\big] P^{K,i}_{s} ds\label{eq:MKi-martingale}
\end{align}is a square integrable martingale with bracket given by \eqref{eq:martingaleP}. A similar computation can be done for $\mathbf{A}^{K,m}$ to obtain \eqref{eq:semimartA}.
\end{proof}

\subsection{Large number of species limit}\label{app_proof}

\begin{proof}[Proof of Proposition~\ref{prop_largenbspecies}]

\noindent Recall that for $n, m\geq 1$, we have:

  \begin{equation}\label{eq_equalityproof}
      \begin{aligned}
  \int_0^1 f(x)d\widetilde\bP^{n,m}_t(x)= &   \int_0^1 f(x)d\widetilde\bP^{n,m}_0(x) \\
  &+\int_0^t \frac{1}{n}\sum_{i=1}^n f(x^i) \Big[g^P\Big( \frac{1}{m}\sum_{j=1}^m mc^{n,m}_{i,j}G_{ij} \widetilde A^j_s\Big)- \frac{1}{n} \sum_{\ell=1}^n k(x^i,x^\ell) \widetilde P^\ell_s \Big]\widetilde P^i_s ds,  \\
 \int_0^1 f(y)d{\widetilde\A}^{n,m}_t(y)= &   \int_0^1 f(y)d\widetilde\A^{n,m}_0(y)  \\
    & +\int_0^t \frac{1}{m}\sum_{j=1}^m f(y^j) \Big[g^A\Big( \frac{1}{n}\sum_{i=1}^n nc^{n,m}_{i,j}G_{ij}\widetilde P^i_s\Big) - \frac{1}{m} \sum_{\ell=1}^m h(y^j,y^\ell) \widetilde A^\ell_s \Big] \widetilde A^j_s\ ds.
   \end{aligned}
  \end{equation}
   where conditionally on $(x^i,y^j)_{i\in \lbrac 1, n\rbrac, j\in \lbrac 1,m\rbrac}$, $G_{ij}\sim Bern(\phi(x^i,y^j))$ and are independent r.v., and $c^{n,m}_{i,j}$ are independent r.v. satisfying Equations~\eqref{def:Cijnm} and \eqref{var:Cijnm}.\\

\noindent \textbf{Step 1: Processes are bounded.}
Since $g^P$ and $g^A$ are bounded from above respectively by $M_P$ and $M_A$ and since the competition terms for the pollinators are non-positive, choosing $f\equiv 1$ in the previous equation~\eqref{eq_equalityproof}, we find that a.s.,
\begin{align}\label{ineq_majinit}
\langle \widetilde\bP^{n,m}_t,1\rangle\leq \langle \widetilde\bP^{n,m}_0,1\rangle e^{M_P t}\quad \mbox{ and }\quad \langle \widetilde\A^{n,m}_t,1\rangle\leq \langle \widetilde\A^{n,m}_0,1\rangle e^{M_A t},
\end{align}
for all $n$ and $m$. Notice that these upper-bounds hold whatever the values of $G_{ij}$ and $c_{ij}^{n,m}$.
By Assumption~\eqref{hyp:boundedCI}, we can define a constant $C(T)<+\infty$ such that almost surely
\begin{equation}\label{majoration_initcond}
\sup_{n,m}\left(\langle \widetilde\bP^{n,m}_0,1\rangle e^{M_P T}+\langle \widetilde\A^{n,m}_0,1\rangle e^{M_A T}\right) < C(T).
\end{equation}
Thus, the processes $\big(\widetilde\bP^{n,m}_t(dx),\widetilde\A_t^{n,m}(dy)\big)_{t\in [0,T]}$ take a.s. their values in $(\mathcal{M}_{\leq C(T)}([0,1]))^2$ which is a compact set.\\

\noindent \textbf{Step 2: Relative compactness in $\Co\Big([0,T],\mathcal{M}_{\leq C(T)}([0,1])^2\Big)$}

Following \cite[Theorem 8.3]{billingsley99} and proof of Theorem 5.3 in \cite{fourniermeleard}, it is sufficient to prove that
\begin{itemize}
\item[(i)] $\sup_{n,m\in \N} \left( \langle \widetilde\bP^{n,m}_0,1\rangle +  \langle \widetilde\A^{n,m}_0,1\rangle \right) $ is bounded almost surely, and
\item[(ii)] for any continuous function $f$ on $[0,1]$,
and for any $\varepsilon>0$,
\begin{align*}
&\lim_{\delta \to 0} \limsup_{n,m\in\N} \mathbb{P}\left( \sup_{|t-s|\leq \delta} |\langle \widetilde\bP^{n,m}_s,f\rangle- \langle \widetilde\bP^{n,m}_t,f\rangle|\geq \varepsilon\right)=0 \\
& \lim_{\delta \to 0} \limsup_{n,m\in\N} \mathbb{P}\left( \sup_{|t-s|\leq \delta}|\langle \widetilde\A^{n,m}_s,f\rangle- \langle \widetilde\A^{n,m}_t,f\rangle|\geq \varepsilon\right)=0 .
\end{align*}
\end{itemize}
Point (i) corresponds to \eqref{hyp:boundedCI}. Let us turn to point (ii). Recall the definition of $\widetilde{R}^{A,i}_t$ in Proposition \ref{prop_lln} and notice that $\widetilde R^{A,i}_t=\frac{1}{m}\sum_{j=1}^m mc^{n,m}_{i,j}G_{ij} \widetilde A^j_t$. Then,
\begin{equation*}
\E[\widetilde R^{A,i}_t]
\leq \frac{1}{m}\sum_{j=1}^m \E[m c^{n,m}_{i,j}] \ C(T)\leq \max_{n,m\in \N^*}\|m c^{n,m}\|_{\infty} \ C(T)<\infty,
\end{equation*}
which is finite from Assumption~\ref{ass_funcc}. Thus, there exists $\tilde C>0$ such that
$\P(\widetilde R^{A,i}>\tilde C)$ is sufficiently small. 
Then $g^P$ is bounded on $[0,\tilde C]$ by some constant $\tilde M$ (since it is a Lipschitz function). Thus, on $\{\tilde R^{A,i}\leq \tilde C\}$, for any positive continuous function $f$ on $[0,1]$ and any $0\leq s\leq t\leq T$, 
\begin{align*}
|\langle \widetilde\bP^{n,m}_t,f\rangle- \langle \widetilde\bP^{n,m}_s,f\rangle| &\leq \int_s^t (\tilde M+\|k\|_{\infty}C(T))\|f\|_{\infty}C(T) du.
\end{align*}
Same computations can be done for the sequence of processes $(\widetilde\A^{n,m})_{m\in \N^*}$. And these are sufficient to conclude point (ii).\\

\noindent \textbf{Step 3: Density of the limiting values.}
Let us then prove that every limiting value of $\big(\widetilde\bP^{n,m},\widetilde\A^{n,m}\big)$, denoted here by $(\bar{\bP},\bar{\A})$, is a process whose time marginals at $t>0$ admit densities $\bar{p}_t$ and $\bar{a}_t$ with respect to the Lebesgue measure on $[0,1]$. To achieve this, we dominate the measures $\widetilde\bP^{n,m}_t(dx)$ and $\widetilde\A_t^{n,m}(dy)$ by measures with densities.
For all positive continuous function $f$ on $[0,1]$, with arguments similar to these of \eqref{ineq_majinit}, we prove that
\begin{equation*}
\begin{aligned}
& \langle \bar\bP_t,f\rangle \leq \langle \bar\bP_0,f\rangle e^{M_P t}=\int_0^1 f(x) \big(p_0(x)e^{M_P t}\big) dx,\quad \text{a.s.} \\
& \langle \bar\A_t,f\rangle \leq \langle \bar\A_0,f\rangle e^{M_A t}=\int_0^1 f(y) \big(a_0(y)e^{M_A t}\big) dy, \quad \text{a.s.} 
\end{aligned}
\end{equation*}
which is sufficient to deduce the existence of densities with respect to Lebesgue measure.\\

\noindent \textbf{Step 4: Equation satisfied by the limiting value.}
Let us fix $t>0$ and $f$ a continuous positive function on $[0,1]$. We define for all $t\in [0,1]$ the map $\psi_t$ on $\Co([0,1],\mathcal{M}_{\leq C(T)}([0,1])^2)$, endowed with the uniform convergence, defined by
\[
\psi_t(\bP,\A)=\begin{pmatrix}
\langle \bP_t,f\rangle-\langle \bP_0,f\rangle - \int_0^t\int_0^1f(x)\left[g^P\left((c\phi)\star\A_s(x)\right)-k\star\bP_s(x) \right]d\bP_s(x)ds\\[0.5em]
\langle \A_t,f\rangle-\langle \A_0,f\rangle - \int_0^t\int_0^1f(x)\left[g^A\left((c\phi)\star\bP_s(x)\right)-h\star\A_s(x) \right]d\A_s(x)ds
\end{pmatrix}.
\]
The continuity of map $\psi_t$ is straightforward once noticing that the weak convergence on $\mathcal{M}_{\leq C(T)}^2([0,1])$ is equivalent to convergence of measures in the Kantorovich–Rubinstein distance {(see for example Theorem 6.9 of \cite{villani2009optimal})}
$$
\mathcal{W}_1(\nu,\mu)=\sup\left\{\int_0^1f(x)d(\nu-\mu)(x)\Big| f:[0,1]\mapsto [0,1] \text{ Lipschitz }, Lip(f)\leq 1 \right\},
$$
since $[0,1]$ is a compact space. The function $\psi_t$ is also bounded since all measures are bounded by $C(T)$.\\
Consider a weak limiting value $(\bar{\bP},\bar{\A})$ of $\big(\widetilde\bP^{n,m},\widetilde\A^{n,m}\big)$, and denote by $\big(\widetilde\bP^{n,m},\widetilde\A^{n,m}\big)$ again (with an abuse of notation) the subsequence converging to $(\bar{\bP},\bar{\A})$. Using the Skorokhod representation theorem, we can assume that these processes are defined on the same probability space and that the convergence is almost sure:
\begin{equation}
    \lim_{n,m\rightarrow +\infty} \sup_{t\leq T} \Big(\mathcal{W}_1\big(\widetilde{\bP}_t^{n,m},\bar{\bP}_t\big)+ \mathcal{W}_1\big(\widetilde{\A}_t^{n,m},\bar{\A}_t\big)\Big)=0,\qquad \mbox{ a.s.}\label{convergence:etape1}
\end{equation}Then, by the properties of $\psi_t$,
\begin{equation}\label{limitpsi}
\E\left[\|\psi_t(\widetilde\bP^{n,m},\widetilde\A^{n,m})\|_1\right] \underset{n,m\to +\infty}{\longrightarrow} \E\left[\|\psi_t(\bar\bP,\bar\A)\|_1\right].
\end{equation}

Our purpose is now to prove that for all $t\in [0,T]$,
\begin{equation}
    \lim_{n,m\rightarrow +\infty}\E\left[\|\psi_t(\widetilde\bP^{n,m},\widetilde\A^{n,m})\|_1\right]=0.\label{eq:but-step4-dur}
\end{equation}This will ensure, in addition with \eqref{limitpsi}, that $\psi(\bar\bP,\bar\A)=(0,0)$ a.s. In other words, the limiting value satisfies Equation~\eqref{edp:cinet} a.s. \\
By {\eqref{eq_equalityproof}, \eqref{ineq_majinit} and} \eqref{majoration_initcond} and denoting by $L^P$ (resp. $L^A$) the Lipschitz {constant} of $g^P$ (resp. $g^A$):
\begin{align}   \E&\left[\|\psi_t(\widetilde\bP^{n,m},\widetilde\A^{n,m})\|_1\right]\nonumber\\
   =& \E\left[\left| \int_0^t \frac{1}{n}\sum_{i=1}^n f(x^i) \widetilde{P}^i_s \bigg(g^P\Big(\frac{1}{m}\sum_{j=1}^m G_{ij}mc^{n,m}_{i,j}\widetilde A^j_s\Big) -g^P\left((c\phi)\star\widetilde\A_s^{n,m}(x^i)\right)\bigg)ds\right|\right]\nonumber\\
   &+\E\left[\left| \int_0^t \frac{1}{m}\sum_{j=1}^m f(y^j) \widetilde{A}^j_s \bigg(g^A\Big(\frac{1}{n}\sum_{i=1}^n G_{ij}nc^{n,m}_{i,j}\widetilde P^i_s\Big) -g^A\left((c\phi)\star\widetilde\bP_s^{n,m}(y^j)\right)\bigg)ds\right|\right],\nonumber\\
   \leq & \int_0^t \frac{C(T)L^P \|f\|_\infty}{n}\sum_{i=1}^n 
   \E\left[\left|\frac{1}{m}\sum_{j=1}^m G_{ij}mc^{n,m}_{i,j}\widetilde A^j_s -(c\phi)\star\widetilde\A_s^{n,m}(x^i)\right|\right]\nonumber\\
 &  \hspace{3cm}+\int_0^t \frac{C(T)L^A \|f\|_\infty}{m}\sum_{j=1}^m 
   \E\left[\left|\frac{1}{n}\sum_{i=1}^n G_{ij}nc^{n,m}_{i,j}\widetilde P^i_s -(c\phi)\star\widetilde\bP_s^{n,m}(y^j)\right|\right]\, ds.\label{term3}
\end{align}
The two terms in \eqref{term3} having similar expressions, we focus on the first one with the aim of showing that uniformly for all $t\in [0,T]$,
\begin{equation}
\lim_{n,m\rightarrow +\infty}\E\left[\left|\frac{1}{m}\sum_{j=1}^m G_{ij}mc^{n,m}_{i,j}\widetilde A^j_t -(c\phi)\star\widetilde\A_t^{n,m}(x^i)\right| \right] =0.  \label{eq:but-step4-dur2}
\end{equation}
The proof of \eqref{eq:but-step4-dur2} is itself divided into smaller steps. To this aim, let us set $\eps>0$ and show that for $n$ and $m$ sufficiently large the term in the left hand side of \eqref{eq:but-step4-dur2} is smaller than $\eps$. The difficulty is that $(\widetilde{A}^j_t)$ are not independent from $(G_{ij})$ and $(c_{ij}^{n,m})$, except at time $t=0$.\\

\noindent \textbf{Step 4.1: \eqref{eq:but-step4-dur2} is satisfied for $t=0$.}
At time $t=0$, we have:
\begin{align}
& \E\left[\left|\frac{1}{m}\sum_{j=1}^m G_{ij}mc^{n,m}_{i,j}\widetilde A^j_0 -(c\phi)\star\widetilde\A_0^{n,m}(x^i)\right|  \right]\nonumber\\
&\hspace{3cm}
\leq
\E\left[\left| \frac{1}{m}\sum_{j=1}^m \left( G_{ij}mc^{n,m}_{i,j} - m\, c^{n,m}(x^i,y^j) \phi(x^i,y^j)\right)\widetilde{A}^j_0 \right|\right]\nonumber\\ 
& \hspace{3,2cm}
+ \E\left[\left|\frac{1}{m}\sum_{j=1}^m \phi(x^i,y^j) \widetilde{A}^j_0 \left( m\, c^{n,m}(x^i,y^j)-c(x^i,y^j)\right)\right|\right].\label{etape4-1}
\end{align}
By \eqref{majoration_initcond}, the second term in the right hand side of \eqref{etape4-1} is upper bounded by $C(T) \|\phi\|_\infty \|mc^{n,m}-c\|_{\infty}$, which converges to zero by {Assumption~\ref{ass_funcc}}.\\
Let us now consider the first term in the right hand side of \eqref{etape4-1}. Since at $t=0$ the random variables $(\widetilde{A}^j_0)$, $(G_{ij})$ and $(c_{ij}^{n,m})$ are independent, we obtain by Assumption \ref{hyp:Gij},  the assumptions of Proposition~\ref{prop_largenbspecies} and \eqref{hyp:cij} and by using Cauchy-Schwarz inequality that
\begin{multline*}
    \E\left[\left| \frac{1}{m}\sum_{j=1}^m \left( G_{ij}mc^{n,m}_{i,j} - m\, c^{n,m}(x^i,y^j) \phi(x^i,y^j)\right)\widetilde{A}^j_0 \right|\right] \\
    \leq  \frac{1}{m}  \E\left[\text{Var}^{1/2}\left( \sum_{j=1}^m \widetilde A^j_0 G_{ij} mc^{n,m}_{i,j} \, \Big|\, x^i,y^j, \widetilde{A}^j_0, \ i\in \lbrac 1,n\rbrac, j\in \lbrac 1,m \rbrac\right)\right].
\end{multline*}
Then, recalling \eqref{var:Cijnm} and the fact that $\text{Var}(XY)=\text{Var}(X)\E[Y^2]+\text{Var}(Y)\E[X]^2$ for any independent r.v. $X$ and $Y$:
\begin{align*}
\E\left[\left| \frac{1}{m}\sum_{j=1}^m \left( G_{ij}mc^{n,m}_{i,j} - m\, c^{n,m}(x^i,y^j) \phi(x^i,y^j)\right)\widetilde{A}^j_0 \right|\right]  \leq & \frac{C(T)}{m}  \E\left[\left(\sum_{j=1}^m \text{Var}_{x,y}\left(   G_{ij} mc^{n,m}_{i,j}\right)\right)^{1/2}\right]\\
\leq & \frac{C(T)}{\sqrt{m}}\left(V_{max}^{1/2}+\|mc^{n,m}\|_\infty \|\phi(1-\phi)\|_{\infty}^{1/2}\right).
\end{align*}
Thus, the expectation in the l.h.s of~\eqref{etape4-1} can be made smaller than $\varepsilon/3$ for $n,m$ larger than some $N_0$, $M_0$.
\medskip

Along the proof of Step 4.1, the independence between the random network ($G_{ij}$'s and $c_{ij}^{n,m}$'s) and the abundancies ($(\widetilde{A}^j_0)$ and $(\widetilde{P}^i_0)$'s) was a crucial ingredient, which does not hold any more for $t>0$. Hence the difficulty. We will first show that \eqref{eq:but-step4-dur} holds on a small time intervals $[0,\delta)$, for some $\delta>0$. For this, we will construct auxiliary processes $(\widehat{A},\widehat{P})$ that do not depend on the $G_{ij}$'s and $c_{ij}^{n,m}$'s. Then, for all $i\in \lbrac 1,n\rbrac$,
\begin{multline}
 \E\left[\left|\frac{1}{m}\sum_{j=1}^m G_{ij}mc^{n,m}_{i,j}\widetilde A^j_t -(c\phi)\star\widetilde\A_t^{n,m}(x^i)\right|  \right]
\leq 
\E\left[\left| \frac{1}{m}\sum_{j=1}^m  G_{ij}mc^{n,m}_{i,j} \Big(\widetilde{A}^j_t -\widehat{A}^j_t\Big) \right| \right.\\
+ \left.\left|\frac{1}{m} \sum_{j=1}^m G_{ij} m \, c_{ij}^{n,m}\widehat{A}^j_t - c\phi \star \widehat{\A}^{n,m}_t(x_i) \right|+ \left|c\phi \star \Big(\widehat{\A}^{n,m}_t - \widetilde{\A}^{n,m}_t\Big)(x_i)\right|\right]. \label{decomposition}
\end{multline}
The second term of~\eqref{decomposition} will be treated using the exact same computation as above and thus it will be upper-bounded by $\eps/3$ as soon as $n,m$ are greater than some threshold $N_0,M_0$ and if the auxiliary processes satisfy~\eqref{majoration_initcond}. Our aim is thus to find auxiliary processes such that the latter is true and such that the first and third terms of~\eqref{decomposition} are small.

Then, we will extend the result by recursion to any $t\in [0,T]$.\\

\noindent \textbf{Step 4.2: Construction of auxiliary processes $\widehat{P}^i$, $\widehat{A}^j$.}
For $n,m$ given, and $t\geq 0$, let us define, 
\begin{equation}
    \widehat{\bP}^{n,m}_t (dx)= \sum_{i=1}^n \widehat{P}^i_t \delta_{x^i},\qquad \widehat{\A}^{n,m}_t(dx)= \sum_{j=1}^m \widehat{A}^j_t \delta_{y^j},
\end{equation}
where for $i\in \lbrac 1, n\rbrac$ and $j\in \lbrac 1, m\rbrac $, the initial state of the auxiliary processes is constructed using a discretization of the limiting initial condition $(a_0(y)dy,p_0(x)dx)$: \begin{equation}\label{def_condinitiale}
\widehat{P}_0^i=n\int_{x^i}^{x^{i+1}}p_0(x)dx \quad \text{and} \quad \widehat{A}_0^j=m\int_{x^j}^{x^{j+1}}a_0(y)dy \quad \text{with }x^{n+1}=y^{m+1}=1.
\end{equation}
Then for all $t\geq 0$, we construct the $\widehat{P}^i$'s and $\widehat{A}^j$'s as the unique solutions of the following system
\begin{align*}
 &   \frac{d\widehat{P}^i_t}{dt}= \widehat{P}^i_t \left(g^P\big(c\phi \star \widehat{\A}^{n,m}_0(x^i)\big)-\frac{1}{n}\sum_{\ell =1}^n k(x^i,x^\ell)\widehat{P}^\ell_t \right)\\
 &   \frac{d\widehat{A}^j_t}{dt}= \widehat{A}^j_t \left(g^A\big(c\phi \star \widehat{\bP}^{n,m}_0(y^j)\big)-\frac{1}{m}\sum_{\ell =1}^m h(y^j,y^\ell)\widehat{A}^\ell_t\right).
\end{align*}
Notice that the processes $(\widehat{P}^i_t, t\geq 0)$ and $(\widehat{A}^j_t, t\geq 0)$ depend on {$c$, $\phi$ and on the sequences $(x^i)_{i\leq n}$ and $(y^j)_{j\leq m}$,} but not on the $G_{ij}$'s and the $c_{ij}^{n,m}$'s since their initial condition does not. \\
Moreover, proceeding as in Step 1 (to obtain~\eqref{majoration_initcond}), we can prove that for all $t\in [0,T]$,
\begin{equation}\label{majhatPA}
\sup_{n,m}( \langle \widehat{\bP}^{n,m}_t,1\rangle + \langle \widehat{\A}^{n,m}_t,1\rangle) \leq C(T),\qquad \mbox{ a.s.},
\end{equation}
hence the second term of~\eqref{decomposition} is smaller than $\eps/3$, as soon as $n,m\geq N_0,M_0$.\\

\noindent \textbf{Step 4.3: the third term of~\eqref{decomposition} is small.}
Notice that, under the assumptions on $c$ and $\phi$, there exists $C_0>0$ such that for all measures $\nu$ and $\mu$, and for all $x\in [0,1]$
\begin{equation*}
    \left|c\phi \star (\nu-\mu)(x)\right|\leq C_0 \mathcal{W}_1(\nu,\mu).
\end{equation*}
Thus, to show this step, we will prove that $\mathcal{W}_1(\widehat{\A}^{n,m}_t ,\widetilde{\A}^{n,m}_t )$ is small.
To this aim, let $f$ be a Lipschitz function such that $\|f\|_{\infty}\leq 1$ and such that its Lipschitz constant $\Lip(f)$ is also bounded by 1, $\Lip(f)\leq 1$. According to the equations satisfied by all the processes, for any $t\geq 0$, 
\begin{align*}
&\left|\langle \widehat{\A}^{n,m}_t - \widetilde{\A}^{n,m}_t, f\rangle\right|- \left|\langle \widehat{\A}^{n,m}_0 - \widetilde{\A}^{n,m}_0, f\rangle\right|\\
&\leq \int_0^t\left|\frac{1}{m}\sum_{j=1}^{m}f(y^j)\left[\widetilde{A}^j_s \left(g^A\big(c\phi \star \widehat{\bP}^{n,m}_0(y^j)\big)-g^A\Big(\frac{1}{n}\sum_{i=1}^n G_{ij} \,m\, c_{ij}^{n,m} \widetilde{P}^i_s\Big)\right)  \right.\right.\\
 & \hspace{1cm}\left. \left.  + g^A\Big(c\phi \star \widehat{\bP}^{n,m}_0(y^j)\Big) \big(\widehat{A}^j_s-\widetilde{A}^j_s\big)-\frac{1}{m}\sum_{\ell =1}^m h(y^j,y^\ell)\big(\widehat{A}^\ell_s \widehat{A}^j_s - \widetilde{A}^\ell_s \widetilde{A}^j_s \big)\right]\right|ds\\
 %\leq & \mathcal{W}_1(\widehat{\A}^{n,m}_0,\widetilde{\A}^{n,m}_0)\\
&\leq \int_0^t\left|C(T)L^A\left( \max_{1\leq j\leq m} \left|\langle \widehat{\bP}^{n,m}_0-\widetilde{\bP}^{n,m}_0, c\phi(.,y^j)\rangle\right|+K_c\max_{1\leq i\leq n}|\widetilde{P}^i_s- \widetilde{P}^i_0|\right)\right.
 + \left|\langle \widehat{\A}^{n,m}_s - \widetilde{\A}^{n,m}_s, \bar{f}\rangle\right|\\
 &\hspace{1cm}\left.  +\frac{C(T)}{m}\sum_{j =1}^m |\langle \widehat{\A}^{n,m}_s- \widetilde{\A}^{n,m}_s, h(y^j,.)\rangle|
  +\frac{C(T)}{m}\sum_{\ell =1}^m |\langle \widetilde{\A}^{n,m}_s- \widehat{\A}^{n,m}_s, fh(.,y^{\ell})\rangle|\right|ds,
\end{align*}
where {$K_c$ corresponds to the bound defined in Assumption~\ref{hyp:cij} and} $\bar{f}(y)=f(y)g^A(c\phi \star \widehat{\bP}^{n,m}_0(y))$ is a Lipschitz continuous and bounded function such that $\Lip(\bar{f})$ is bounded independently from $n,m$, according to the assumptions on $c$ and $\phi$. Using also that $h$ is a Lipschitz function, we finally deduce that there exists $C_1>0$ such that
\begin{align*}
\left|\langle \widehat{\A}^{n,m}_t - \widetilde{\A}^{n,m}_t, f\rangle\right|\leq & \mathcal{W}_1(\widehat{\A}^{n,m}_0,\widetilde{\A}^{n,m}_0)\\
&+ \int_0^t C_1\left( \mathcal{W}_1(\widehat{\bP}^{n,m}_0,\widetilde{\bP}^{n,m}_0)+\max_{1\leq i\leq n}|\widetilde{P}^i_s- \widetilde{P}^i_0|+\mathcal{W}_1(\widehat{\A}^{n,m}_s,\widetilde{\A}^{n,m}_s)\right)ds.
\end{align*}
Taking the supremum in the left hand side w.r.t. $f$, and using Gronwall's lemma, we deduce that for all $t\geq 0$
\begin{equation}\label{eq_gronwall}
\mathcal{W}_1(\widehat{\A}^{n,m}_t,\widetilde{\A}^{n,m}_t)\leq \left(\mathcal{W}_1(\widehat{\A}^{n,m}_0,\widetilde{\A}^{n,m}_0)+C_1t \mathcal{W}_1(\widehat{\bP}^{n,m}_0,\widetilde{\bP}^{n,m}_0)+C_1t\max_{1\leq i\leq n,s\leq t}|\widetilde{P}^i_s- \widetilde{P}^i_0|\right)e^{C_1 t}.
\end{equation}
On the one hand, for the third term in the r.h.s., we have for any $s\in [0,T]$,
\begin{align*}
\big|\widetilde{P}^i_s -\widetilde{P}^i_0 \big| = & \left|\int_0^s \left(g^P\Big(\frac{1}{m}\sum_{j=1}^m G_{ij} m \, c_{ij}^{n,m} \widetilde{A}^j_u\Big) - \frac{1}{n}\sum_{\ell=1}^n k(x^i,x^\ell) \widetilde{P}^\ell_u\right)\, \widetilde{P}^i_u \, du\right|\\
\leq & s\, C(T) \big(|g^P(0)|+ L^P C(T) K_c + \|k\|_\infty C(T)\big)=sC_2,
\end{align*}
where we set $C_2:=C(T) \big(|g^P(0)|+ L^P C(T) K_c + \|k\|_\infty C(T)\big)$, which is a constant independent from $n,m$ and $s$. On the other hand, $\widetilde{\bP}^{n,m}_0$ and $\widetilde{\A}^{n,m}_0$ converge weakly towards $p_0(x)dx$ and $a_0(y)dy$ in probability. Moreover the sequence of r.v. $\mathcal{W}_1(\widehat{\A}^{n,m}_0,a_0(y)dy)$ and $\mathcal{W}_1(\widehat{\bP}^{n,m}_0,p_0(x)dx)$ are bounded in $L^2$, and using the dominated convergence theorem, we deduce the convergences of the expectations to $0$. In addition with the definition of $\widehat{\bP}^{n,m}_0$ and $\widehat{\A}^{n,m}_0$ (see \eqref{def_condinitiale}), we can find $N_1, M_1\geq N_0,M_0$ such that for all $n,m\geq N_1,M_1$,
\begin{equation}\label{ass_condinitiale}
\E[\mathcal{W}_1(\widehat{\A}^{n,m}_0,\widetilde{\A}^{n,m}_0)]\leq \frac{\eps}{12C_0}\quad  \text{ and } \quad\E[\mathcal{W}_1(\widehat{\bP}^{n,m}_0,\widetilde{\bP}^{n,m}_0)]\leq \frac{\eps}{12C_0}.
\end{equation}
In addition with \eqref{eq_gronwall}, we deduce that there exists $\delta_1$ such that for all $t\leq \delta_1$, all $x\in[0,1]$ and all $n,m\geq N_1,M_1$,
\begin{align}
\E\left|c\phi\star (\widehat{\A}^{n,m}_t-\widetilde{\A}^{n,m}_t)(x)\right|&\leq C_0\E\left[\mathcal{W}_1(\widehat{\A}^{n,m}_t,\widetilde{\A}^{n,m}_t)\right]\nonumber\\
&\leq \left(\frac{\eps}{12}+C_1\delta_1 \frac{\eps}{12}+C_0C_1C_2\delta_1^2 \right)e^{C_1 \delta_1}
\leq \frac{\eps}{3}.\label{term3step4}
\end{align}

Using the same computations, we can find $\delta_0\leq \delta_1$ such that for all $\delta\leq \delta_0$ and all $n,m\geq N_1,M_1$, the same result holds for the processes $\widehat{\bP}^{n,m}$  and $\widetilde{\bP}^{n,m}$.\\

\noindent \textbf{Step 4.4: the first term of~\eqref{decomposition} is small.} We will use similar computations as in Step 4.3 to reach this step. To this aim, we define for all $t\geq 0$
$$
\Sigma^{A,n,m}(t%\widetilde{\A}^{n,m}_t,\widehat{\A}^{n,m}_t
)=\sup\left\{\frac{1}{m}\sum_{j=1}^mG_{ij}mc^{n,m}_{i,j}f(y^j) \Big(\widetilde{A}^j_t -\widehat{A}^j_t\Big)\Big| f:[0,1]\mapsto [0,1]\text{ Lipschitz, } \Lip(f)\leq 1 \right\},
$$
and a similar definition for $\Sigma^{P,n,m}(t)$. Let $f$ be a Lipschitz function, bounded by $1$. Then
\begin{align*}
   & \left| \frac{1}{m}\sum_{j=1}^m  G_{ij}mc^{n,m}_{i,j}f(y^j) \Big(\widetilde{A}^j_t -\widehat{A}^j_t\Big)\right|-\left| \frac{1}{m}\sum_{j=1}^m  G_{ij}mc^{n,m}_{i,j}f(y^j) \Big(\widetilde{A}^j_0 -\widehat{A}^j_0\Big)\right|\\
   &\leq \int_0^t\left|\frac{1}{m}\sum_{j=1}^m  G_{ij}mc^{n,m}_{i,j}f(y^j)g^A\Big(c\phi\star \widetilde{\bP}^{n,m}_0(y^j)\Big)\Big(\widetilde{A}^j_s -\widehat{A}^j_s\Big)\right. \\
   &\hspace{0.5cm}+ \frac{1}{m}\sum_{j=1}^m  G_{ij}mc^{n,m}_{i,j}f(y^j)\widetilde{A}^j_s \left(g^A\Big(\frac{1}{n}\sum_{i=1}^n G_{ij} \, n\, c_{ij}^{n,m} \widetilde{P}^i_s\Big)-g^A\Big(c\phi\star \widetilde{\bP}^{n,m}_0(y^j)\Big)\right)\\
   &\hspace{0.5cm}+ \frac{1}{m}\sum_{\ell=1}^m \widehat{A}^\ell_s \left(\frac{1}{m}\sum_{j=1}^m G_{ij}mc^{n,m}_{i,j} f(y^j)h(y^j,y^\ell)\Big(\widehat{A}^j_s-\widetilde{A}^j_s\Big)\right) \\
   &\hspace{0.5cm}\left. + \frac{1}{m}\sum_{j=1}^m G_{ij}mc^{n,m}_{i,j}f(y^j) \widetilde{A}^j_s \frac{1}{m}\sum_{\ell=1}^m h(y^j,y^\ell)\Big(\widehat{A}^\ell_s-\widetilde{A}^\ell_s\Big) \right|ds.
\end{align*}
Following the same argument as in the previous step, using Gronwall's lemma, we deduce that 
\begin{equation}\label{gronwallbis}
    \Sigma^{A,n,m}(t)\leq \left(\Sigma^{A,n,m}(0)+C_1t \Sigma^{P,n,m}(0)+C_1tC_2\right)e^{C_1 t},
\end{equation}
where $C_1$ and $C_2$ have been defined in the previous step.
\medskip

It remains to prove that $\E[\Sigma^{A,n,m}(0)]$ and $\E[\Sigma^{P,n,m}(0)]$ converge to $0$ with $n$ and $m$. To this aim, we follow the exact same computations as in Step 4.1 to deduce that for any Lipschitz function $f$ with $\|f\|_{\infty}\leq 1$, 
\begin{multline*}
\E\left|\frac{1}{m}\sum_{j=1}^m G_{ij}mc^{n,m}_{i,j}f(y^j) A^j_0 -(c\phi f)\star \A_0^{n,m}(x^i)\right|  \leq \frac{C(T)}{\sqrt{m}}\left(V_{max}^{1/2}+\|mc^{n,m}\|_\infty \|\phi(1-\phi)\|_{\infty}^{1/2}\right)\\
+C(T) \|\phi\|_\infty \|mc^{n,m}-c\|_{\infty},
\end{multline*}
with $A^j_0=\widetilde{A}^j_0$ or $\widehat{A}^j_0$. In addition with the weak convergence of $\widetilde{\A}_0^{n,m}$ and $\widehat{\A}_0^{n,m}$ in probability towards the same limits, and with similar computations for $P$, we deduce that there exists $N_2,M_2\geq N_1,M_1$ such that 
\begin{equation*}
    \E[\Sigma^{A,n,m}(0)]\leq \frac{\eps}{12 C_0}\quad \text{and} \quad  \E[\Sigma^{P,n,m}(0)]\leq \frac{\eps}{12 C_0}.
\end{equation*}
Using also \eqref{gronwallbis}, and remembering the definition of $\delta_0$ at the end of Step 4.2, we deduce that for all $n,m\geq N_2,M_2$ and all $t\leq \delta_0$,
\begin{equation}
    \E\left| \frac{1}{m}\sum_{j=1}^m  G_{ij}mc^{n,m}_{i,j} \Big(\widetilde{A}^j_t -\widehat{A}^j_t\Big) \right| \leq C_0\E\left[ \Sigma^{A,n,m}(t)\right]\leq \frac{\eps}{3}. \label{term1step4}
\end{equation}
Moreover, the same result holds for the processes $\widehat{\bP}^{n,m}$  and $\widetilde{\bP}^{n,m}$.\\

\noindent \textbf{Step 4.5: conclusion for $t\leq \delta_0$.}\\
Combining Equations~\eqref{decomposition}, \eqref{term1step4}, \eqref{term3step4} and the end of Step 4.2, we conclude that the limit in \eqref{eq:but-step4-dur2} is true uniformly for $t\in[0,\delta_0]$. More precisely, for all $n,m\geq N_2, M_2$ and all $t\leq \delta_0$,
$$
\E\left[\left|\frac{1}{m}\sum_{j=1}^m G_{ij}mc^{n,m}_{i,j}\widetilde A^j_t -(c\phi)\star\widetilde\A_t^{n,m}(x^i)\right| \right]\leq \eps. 
$$
Finally, \eqref{eq:but-step4-dur} is valid for all $t\leq \delta_0$ and any limiting value satisfies the equation~\eqref{edp:cinet} for $t\in[0,\delta_0]$. According to the uniqueness of the solution to this equation (see Step 5 below), we deduce in particular that 
\begin{equation}
  \lim_{n,m\rightarrow +\infty} \Big(\mathcal{W}_1\big(\widetilde{\bP}_{\delta_0}^{n,m},\bar{\bP}_{\delta_0}\big)+ \mathcal{W}_1\big(\widetilde{\A}_{\delta_0}^{n,m},\bar{\A}_{\delta_0}\big)\Big)=0,\qquad \mbox{ in probability,}
\end{equation}
with $(\bar{\bP}_t,\bar{\A}_t)_{t\in [0,\delta_0]}$ solution to~\eqref{edp:cinet}. According to Step 3, $\bar{\bP}_t$ and $\bar{\A}_t$ admit densities $\bar{p}_t$ and $\bar{a}_t$ with respect to the Lebesgue measure. \\

\noindent \textbf{Step 4.6: recursion.}
We proceed now by recursion, by defining new auxiliary processes, for $n,m$ given,  $t\geq \delta_0$,
\begin{equation*}
    \widehat{\bP}^{n,m}_t (dx)= \sum_{i=1}^n \widehat{P}^i \delta_{x^i},\qquad \widehat{\A}^{n,m}_t(dx)= \sum_{j=1}^m \widehat{A}^j \delta_{y^j},
\end{equation*} 
with for $i\in \lbrac 1, n\rbrac$ and $j\in \lbrac 1, m\rbrac$, $\widehat{P}^i_{\delta_0}$ and $\widehat{A}^j_{\delta_0}$ are defined in a similar way as the beginning of Step 4.2 using $\bar{p}_{\delta_0}(x)dx$ and $\bar{a}_{\delta_0}(y)dy$ instead of ${p}_{0}(x)dx$ and ${a}_{0}(y)dy$:
\begin{equation}
    \widehat{P}_{\delta_0}^i=n\int_{x^i}^{x^{i+1}}p_{\delta_0}(x)dx \quad \text{and} \quad \widehat{A}_{\delta_0}^j=m\int_{x^j}^{x^{j+1}}a_{\delta_0}(y)dy \quad \text{with }x^{n+1}=y^{m+1}=1.
\end{equation}
And for all $t\geq \delta_0$, we set
\begin{align*}
 &   \frac{d\widehat{P}^i_t}{dt}= \widehat{P}^i_t \left(g^P\big(c\phi \star \widehat{\A}_{\delta_0}(x^i)\big)-\frac{1}{n}\sum_{\ell =1}^n k(x^i,x^\ell)\widehat{P}^\ell_t \right)\\
 &   \frac{d\widehat{A}^j_t}{dt}= \widehat{A}^j_t \left(g^A\big(c\phi \star \widehat{\bP}_{\delta_0}(y^j)\big)-\frac{1}{m}\sum_{\ell =1}^m h(y^j,y^\ell)\widehat{A}^\ell_t\right).
\end{align*}
Notice again that the processes $(\widehat{P}^i_t)$ and $(\widehat{A}^j_t)$ do depend on {$c$, $\phi$ and on the sequences $(x^i)_{i\leq n}$ and $(y^j)_{j\leq m}$,} but not on the $G_{ij}$'s and the $c_{ij}^{n,m}$'s. Thus, we can proceed as in Step 4.2, 4.3, {4.4 and 4.5 and find $N_4, M_4\geq N_2,M_2$ such that for all $n,m\geq N_4, M_4$ and all $t\in [\delta_0,2\delta_0]$}, 
\begin{equation}\label{eqstep4fin}
\E\left[\left|\frac{1}{m}\sum_{j=1}^m G_{ij}mc^{n,m}_{i,j}\widetilde A^j_t -(c\phi)\star\widetilde\A_t^{n,m}(x^i)\right| \right]\leq \eps. 
\end{equation}
In other words, we can conclude similarly as in {Step 4.5} for all $t\in [0,2\delta_0]$. With a direct recursive argument, we finally be able to find $N_{2\kappa}, M_{2\kappa}$, with $\kappa=\lceil T/\delta_0 \rceil$, such that for all $n,m\geq N_{2\kappa}, M_{2\kappa}$ and all $t\in [0,T]$, \eqref{eqstep4fin} holds.

 This finally proves {\eqref{eq:but-step4-dur2} and thus \eqref{eq:but-step4-dur}, which} finishes the proof of Step 4. \\

\noindent \textbf{Step 5: Uniqueness of the solution of \eqref{edp:cinet}.}
It remains to prove the uniqueness of the solution to~\eqref{edp:cinet} to conclude. To this aim, let us take $(\bP^1,\A^1)$ and $(\bP^2,\A^2)$ two deterministic measures, weak solutions to \eqref{edp:cinet} and with identical initial conditions. Let us denote by $\rho$ the Radon metric between two measures $\nu$ and $\mu$ on $\mathcal{M}_{F}([0,1])^2$:
\[
\rho(\nu,\mu) = \sup\left\{ \int_0^1f(x)d(\nu-\mu)(x) \Big| f:[0,1]\mapsto [-1,1] \text{ continuous} \right\}.
\]
With straightforward computations, for any continuous function $f:[0,1]\mapsto[-1,1]$, we can find some finite constants $C_1, C_2$, independent from $t$ and $f$, such that
\begin{align*}
    |\langle \bP^1_t-\bP^2_t,f \rangle| &= \left|\int_0^t\int_0^1f(x)\left[g^P(c\phi\star\A^1_s(x))-k\star\bP^1_s\right]d\bP^1_s(x)ds\right.\\
    &\qquad\left.-\int_0^t\int_0^1f(x)\left[g^P(c\phi\star\A^2_s(x))-k\star\bP^2_s\right]d\bP^2_s(x)ds\right|\\
    &\leq C_1 \int_0^t\rho(\bP^1_s,\bP^2_s)ds+\int_0^t\int_0^1\left(L^P\left|c\phi\star(\A^1_s-\A^2_s(x))\right|+\left|k\star(\bP^1_s-\bP^2_s)(x)\right|\right)d\bP^2_s(x)ds\\
    &\leq C_2 \int_0^t(\rho(\bP^1_s,\bP^2_s)+\rho(\A^1_s,\A^2_s))ds,
\end{align*}
since $x\mapsto f(x)(g^P(c\phi\star\A^1_s(x))-k\star\bP^1_s(x))$ are continuous bounded functions on $[0,1]$ for all $s\in [0,T]$, and $k$ and $c\phi$ are continuous bounded functions on $[0,1]^2$. Same computations can be done with $\A^1$ and $\A^2$, then taking the supremum over all $f$, we find a constant $C_3>0$ such that
\begin{equation*}
    \rho(\bP^1_t,\bP^2_t)+\rho(\A^1_t,\A^2_t)\leq C_2 \int_0^t(\rho(\bP^1_s,\bP^2_s)+\rho(\A^1_s,\A^2_s))ds.
\end{equation*}
We conclude with Gronwall inequality that $\rho(\bP^1_t,\bP^2_t)+\rho(\A^1_t,\A^2_t)=0$ for all $t\in [0,T]$. In other words, $(\bP^1,\A^1)$ and $(\bP^2,\A^2)$ are identical, and the solution to \eqref{edp:cinet} is unique.
This ends the proof.
\end{proof}

{\footnotesize
\providecommand{\noopsort}[1]{}\providecommand{\noopsort}[1]{}

}

\end{document}